\documentclass[11pt, a4paper]{amsart}
\usepackage{amsfonts,amssymb,amsmath, amsthm}
\usepackage{mathrsfs}
\usepackage{lmodern}
\usepackage{latexsym}
\usepackage{bm}
\usepackage{enumitem}

\newtheorem{thm}{Theorem}[section]
\newtheorem{lem}[thm]{Lemma}

\theoremstyle{defn}
\newtheorem{defn}[thm]{Definition}

\numberwithin{equation}{section}
\usepackage[plainpages=false,pdfpagelabels,backref=page,citecolor=red]{hyperref}
\usepackage{xcolor}
\hypersetup{
colorlinks,
linkcolor={cyan!90!black},
citecolor={magenta},
urlcolor={green!40!black}
} %important to load after amsthm and before the rest
\newcommand{\R}{\mathbb{R}}

\newcommand{\bgs}[1]{\begin{equation*} \begin{aligned}#1\end{aligned}\end{equation*}} % Equation with aligned
\newcommand{\syslab}[2] []  {\begin{equation}#1  \left\{\begin{aligned}#2\end{aligned}\right.\end{equation}}
\renewcommand{\d}{\, \mathrm{d}} %differential
\newcommand{\eps}{\varepsilon} %epsilon

\def\Xint#1{\mathchoice
{\XXint\displaystyle\textstyle{#1}}%
{\XXint\textstyle\scriptstyle{#1}}%
{\XXint\scriptstyle\scriptscriptstyle{#1}}%
{\XXint\scriptscriptstyle%
\scriptscriptstyle{#1}}%
\!\int}
\def\XXint#1#2#3{{\setbox0=\hbox{$#1{#2#3}{%
\int}$ }
\vcenter{\hbox{$#2#3$ }}\kern-.6\wd0}}
\def\barint{\,\Xint -} % \, corrects the \! used in the definition
\def\bariint{\barint_{} \kern-.4em \barint}
\def\bariiint{\bariint_{} \kern-.4em \barint}
\renewcommand{\iint}{\int_{}\kern-.34em \int} %\, minor space between the integrals
\renewcommand{\iiint}{\iint_{}\kern-.34em \int} %\, minor space between the integrals
\newtheorem*{remark}{Remark} %not number remarks

\title[Asymptotic mean value properties for space-time fractional parabolic operators]{Asymptotic mean value formulas, nonlocal space-time parabolic operators and anomalous tug-of-war games}

\address{Carmina Fjellstr\"{o}m\\Department of Mathematics, Uppsala University\\
S-751 06 Uppsala, Sweden}
\email{carmina.fjellstrom@math.uu.se}

\address{Kaj Nystr\"{o}m\\Department of Mathematics, Uppsala University\\
S-751 06 Uppsala, Sweden}
\email{kaj.nystrom@math.uu.se}

\address{Yuqiong Wang\\Department of Mathematics, Uppsala University\\
S-751 06 Uppsala, Sweden}
\email{yuqiong.wang@math.uu.se}

\thanks{K.N was partially supported by grant  2017-03805 from the Swedish research council (VR)}

\author{C. Fjellstr\"{o}m,  K. Nystr{\"o}m and Y. Wang}

\date{\today}

\begin{document}
\begin{abstract}
The fractional heat operator  $(\partial_t-\Delta_x)^s$ and  Continuous Time Random Walks (CTRWs) are interesting and sophisticated mathematical models that can describe complex anomalous systems. In this paper, we  prove asymptotic mean value representation formulas for functions with respect to $(\partial_t-\Delta_x)^s$ and
we introduce new nonlocal, nonlinear parabolic operators related to a tug-of-war which accounts for waiting times and space-time couplings. These nonlocal, nonlinear parabolic operators and equations can be seen as nonlocal versions of the evolutionary infinity Laplace operator.
\end{abstract}

\maketitle

\section{Introduction}

Recently, there has  been a surge in the study of the fractional Laplacian $(-\Delta_x)^s$, as well as of more general linear and nonlinear fractional operators, as mathematicians and scientists have been exploring fractional calculus as a tool for developing more sophisticated mathematical models that can accurately describe complex anomalous systems. In particular, the fractional Laplacian has been used in place of the standard Laplacian in many applications in finance, biology, the sciences and engineering. It is well known that the fractional Laplacian in $\mathbb R^n$ can be defined in many equivalent ways connecting the fractional Laplacian to singular integrals and harmonic analysis,  generalized Dirichlet to Neumann operators, degenerate elliptic equations and spectral theory, and the theory of L\'{e}vy/stable processes. We refer to \cite{lischke} for a recent and comprehensive survey of these connection and applications.

One natural parabolic extension of $(-\Delta_x)^s$ is the parabolic operator $\partial_t+(-\Delta_x)^s$ which appears, for example, in the study of stable processes and in option pricing models, see \cite{caffarelli2007} and the references therein. Other generalization are the time-fractional diffusion operators $\partial_t^\beta+(-\Delta_x)$ and $\partial_t^\beta+(-\Delta_x)^s$ being the sum of a fractional and nonlocal time derivative as well as a local or nonlocal operator in space. These operators have attracted considerable  interest during the last years, mostly due to their applications in the modelling
of anomalous diffusion, see for example \cite{KSVZ}, \cite{KSZ}, and  the references therein.

Decisive progress in the study of several problems involving the operator $(-\Delta_x)^s$ has been achieved through an extension technique, rediscovered in \cite{caffarelli2007}, based on which the fractional Laplacian can be studied through a local but degenerate elliptic operator having degeneracy determined by an $A_2$-weight. In \cite{NS}, the second author, together with O. Sande, established the parabolic analogue of  the result in \cite{caffarelli2007}  by proving that if $s\in (0,1)$, then the nonlocal fractional heat operator $(\partial_t- \Delta_x )^{s}$ can be realized as a Dirichlet to Neumann map for the (local) linear degenerate parabolic operator $\lambda^{1-2s}\partial_t-\nabla_{\lambda,x}\cdot(\lambda^{1-2s}\nabla_{\lambda,x})$,  see also \cite{B-DLC-S,GFT, LN,Stinga-Torrea-SIAM}. As one consequence, local properties of solutions to $( \partial_t- \Delta_x )^{s}u=0$ in a domain $\Omega\times J\subset\mathbb R^n\times \mathbb R$ can be deduced by developing corresponding results for linear degenerate parabolic equations. Formally, the linear degenerate parabolic equation  introduced in  \cite{NS} coincides, when applied to functions which are independent of time, with the linear degenerate elliptic equation introduced in \cite{caffarelli2007}.

The purpose of this paper is to develop certain asymptotic mean value representation formulas reflecting the kernel of the operator $\mathcal{H}^s$, where $\mathcal{H}:=(\partial_t-\Delta_x)$ and to introduce  new nonlocal, nonlinear parabolic operators $\mathcal H^{s,\pm}_\infty$ related to a tug-of-war which accounts for waiting times and space-time couplings. We also study the asymptotic properties of these operators as $s\nearrow 1$.

Our interest in the operator $\mathcal{H}^s$ stems not only from the interesting analysis of these operators and their favourable connections to (local) linear degenerate parabolic operators, but also from the relevance of these operators to the study of anomalous diffusions. Indeed, from a probabilistic point of view,  fractional powers of parabolic operators like the heat operator is closely connected to Continuous Time Random Walks (CTRWs), see \cite{Metzler-Klafter} for a more physically oriented introduction and \cite{MB15,BMM05,BMS04,MS06,Meerschaert2012,MS14} for mathematical treatments. CTRWs are a kind of stochastic {process} in which random jumps in space are {coupled} with random waiting times. In general, the scaling limit of  CTRWs are time-changed Markov processes in $\mathbb R^n$ and the clock process is the hitting time of an increasing L\'{e}vy process. This is in contrast to operators like $\partial_t+(-\Delta_x)^s$ and $\partial_t^\alpha +(-\Delta_x)^s$, where the jumps in space are {independent} of the waiting times.

The introduction of the operators $\mathcal H^{s,\pm}_\infty$ is partially  inspired by \cite{bjor}, where the authors consider,  motivated by the tug-of-war game introduced and studied in \cite{PSSW}, a nonlocal version of the game which loosely can be defined as follows: at every step, two players respectively pick a direction, and then, instead of flipping a coin in order to decide which direction to choose and then moving a fixed amount $\epsilon>0$ (as is done in the classical case), it is a $s$-stable L\'{e}vy process that decides the direction (either the one picked by player I or the one picked by player II) and the distance to travel. In this context, they introduce what they call the infinity fractional Laplacian $\Delta_\infty^s$, which is the nonlocal, nonlinear operator underlying the game.

In contrast to \cite{bjor}, we in this paper instead consider a space-time coupled tug-of-war game loosely defined as follows: two players play a zero-sum game and, at each step of the game, the players pick directions $y\in S^{n}$ and $z\in S^{n}$ on the unit sphere, respectively. Then a random waiting time $\tau$, and a one-dimensional position change $\eta$, are drawn from a distribution $K_{1,s}(\eta,\tau)$ conditioned on the event that $(\eta,\tau)\in (\mathbb R^{n}\times\mathbb R_+)\setminus C_\epsilon^+(0,0)$, where $C_\epsilon^+(0,0)$ is the upper half of the (parabolic) cylinder centered at (0,0) with radius $\epsilon$ in space and  height $2\epsilon^2$ in time. At $t+\tau$, the token is moved in space as $x\to x+\eta y$ if $\eta>0$ and as $x\to x+\eta z$ if $\eta<0$.  By the Dynamic Programming Principle (DPP), and heuristically, the fair value of the game will satisfy $\mathcal H^{s,+}_\infty u=0$ in the limit as $\epsilon\to 0$.

It is interesting to consider the limit of the operators introduced as $s\nearrow 1$ and we prove that
\begin{align}\label{lim1}
     \lim_{s\nearrow 1}\mathcal{H}^{s}u(x,t)&=(\partial_t-\Delta_x)u(x,t),
          \end{align}
          and if $\nabla_xu(x,t)\neq 0$, then
\begin{align}\label{lim2}
     \lim_{s\nearrow 1}\mathcal{H}_\infty^{s,\pm}u(x,t)&=(\partial_t\pm\Delta_{\infty,x}^N) u(x,t)\notag\\
     &:=\partial_tu(x,t)\pm|\nabla_xu(x,t)|^{-2}\langle \nabla^2_xu(x,t)\nabla_xu(x,t),\nabla_xu(x,t)\rangle.
          \end{align}
          Here, $\Delta_{\infty,x}^N$ is the (normalized) infinity Laplace operator occurring in the context of the tug-of-war game introduced and studied in \cite{PSSW}, and it is interesting to note that in \cite{manfredipr10c}, it is proved that a function $u=u(x,t):\mathbb R^n\times \mathbb R\to \mathbb R$ is a viscosity solution to the equation
\begin{equation}\label{meanvalue1heatgha}
\partial_tu(x,t)-\Delta_{\infty,x}^Nu(x,t)=0,
\end{equation}
if and only if
\begin{align}\label{meanvalue3aheatha}
u(x,t)&=\frac 1 2\barint_{t-\epsilon^2}^t\biggl ( \biggl\{\max_{\tilde x\in\overline{{B_\epsilon(x)}}}u(\tilde x,\tilde t)\biggr \}+\biggl\{\min_{\tilde x\in\overline{{B_\epsilon(x)}}}u(\tilde x,\tilde t)\biggr \}\biggr )\d \tilde t+{o}(\epsilon^2),\mbox{ as }\epsilon\to 0,
\end{align}
in the viscosity sense. Through the DPP, this clearly connects the tug-of-war game in \cite{PSSW} to asymptotic mean value formulas.

%%%%%%%%%%%%%%%%%%Write last

The paper is organized as follows. In Section \ref{sec0}, we introduce the fractional heat operator $\mathcal{H}^s$, and briefly discuss the connection to the extension problem and continuous time random walks (CTRWs). In this section, we also state the asymptotic mean value formulas in Theorem \ref{thm1-} and Theorem \ref{thm1a}, and we discuss probabilistic interpretations of these results. We introduce the nonlocal, nonlinear parabolic operators $\mathcal H^{s,\pm}_\infty$ in Section \ref{sec1}, as well as state the asymptotic mean value formula for these operators in Theorem \ref{thm1aba}. In Section \ref{sec2}, we detail the viscosity formalism for $\mathcal{H}^s$, while in Section \ref{sec3}, we prove Theorem \ref{thm1-} and Theorem \ref{thm1a}, where the proof of Theorem \ref{thm1-} is included for completion. In Section \ref{sec4}, we give a heuristic derivation of $\mathcal{H}^{s,\pm}_\infty$ and
introduce an appropriate notion of viscosity solutions. In Section \ref{sec5}, we prove Theorem \ref{thm1aba}, and in Section \ref{sec6}, we examine the limits of the operators as $s\nearrow 1$, where we prove that $\mathcal H^{s}$ does indeed converge to the heat operator $(\partial_t-\Delta_x)$, and $\mathcal H^{s,\pm}_\infty$ converges to the forward/backward evolutionary normalized infinity Laplace operator. Finally, in Section \ref{sec7}, we outline open problems and future research topics devoted to the operators $\mathcal H^{s,\pm}_\infty$, which we hope will stimulate further developments.

%%%%%%%%%%%%%%%%%%%%%%%%%%%%%%%%%%%%%%%%%%%%%%%

\section{The fractional heat operator}\label{sec0}

For $s\in (0,1)$, the fractional heat operator $\mathcal{H}^su(x,t)$
of a function $u=u(x,t):\R^{n+1}\to\R$, $n\geq1$, is given as
\begin{align}\label{eq0}\widehat{\mathcal{H}^su}(\xi,\tau):=(i\tau+|\xi|^2)^s\widehat{u}(\xi,\tau),\ \widehat{u}(\xi,\tau):=\iint_{\mathbb R^{n+1}}e^{i(x\cdot\xi+t\tau)}u(x,t)\, \d x\d t,
\end{align}
for $(\xi,\tau)\in\R^n\times \R$. Using  \cite{Sa}, it follows that if $u\in\mathcal{S}(\mathbb R^{n+1})$, then $\mathcal{H}^su$ can be realized as a parabolic hypersingular integral,
\begin{align}\label{eq1}
     \mathcal{H}^s u(x,t)&=\iint_{\mathbb R^{n+1}} {(u(x,t)-u(x+w,t-\tau))}K_{n,s}(w,\tau)\, \d w\d \tau\notag\\
     &=\iint_{\mathbb R^{n+1}} {(u(x,t)-u(x+w,t+\tau))}K_{n,s}(w,-\tau)\, \d w\d \tau,
     \end{align}
     with kernel
 \begin{eqnarray}\label{eq2}
   K_{n,s}(w,\tau):=K_{n,s}(w,\tau,0,0),\  K_{n,s}(x,t,y,\gamma):=\frac {1}{\Gamma(-s)}\frac {W_{n}(x,t,y,\gamma)}{(t-\gamma)^{1+s}}.
     \end{eqnarray}
Here,
  \begin{eqnarray}\label{eq2+}
    W_n(x,t,y,\gamma):=\frac {1}{(4\pi(t-\gamma))^{n/2}}e^{-\frac{|x-y|^2}{(4(t-\gamma))}}\chi_{(0,\infty)}(t-\gamma),
     \end{eqnarray}
and  $\Gamma(-s)$ is the gamma function evaluated at $-s$. Note that $W_n(x,t,y,\gamma)$ is the heat (transition) kernel in $\mathbb R^n$. Using the symmetry of the kernel $K_{n,s}$, we also have  the representation
\begin{align}\label{eq1asymp+intr}
     \mathcal{H}^s u(x,t)&=\frac 1 2 \iint_{\mathbb R^{n+1}} {(2u(x,t)-u(x+w,t+\tau)-u(x-w,t+\tau))}K_{n,s}(w,-\tau)\, \d w\d \tau,
\end{align}
and we note that by the definition of $K_{n,s}(w,-\tau)$, only points $(w,\tau)\in\mathbb R^n\times \mathbb R_-$ contribute to the integral.

\subsection{The fractional heat operator and the extension problem}
In \eqref{eq0} and \eqref{eq1}, $\mathcal{H}^s$ is introduced through the Fourier transform and identified with a singular integral operator. \eqref{eq1} can also be formulated in a slightly different manner if we introduce
    \begin{equation}\label{eq3}
{P}_\tau u(x,t):=\int\limits_{\mathbb R^n}W_n(x,\tau,w,0)u(w,t-\tau)\, \d w, \ \tau>0.
\end{equation}
Then, by adapting the ideas
first introduced by Balakrishnan \cite{balakrishnan1959,balakrishnan1960} for the fractional powers of closed operators on Banach spaces, $\mathcal{H}^s$ can also be expressed as
        \begin{eqnarray}\label{eq4-}
\mathcal{H}^s u(x,t):= -\frac s{\Gamma(1-s)}\int\limits_0^\infty \tau^{-1-s}\bigl (P_\tau u(x,t)-u(x,t)\bigr )\,  {d\tau}.
    \end{eqnarray}
This representation in \eqref{eq4-} is interesting as it is the vehicle based on which a (local) linear degenerate parabolic equation associated with $\mathcal{H}^s$ can be
    introduced.  Indeed, the essence of Theorem 1 in \cite{NS} is that if we, for $\lambda>0$, introduce
          \begin{eqnarray}\label{extension}
U(\lambda,x,t):=\frac {1}{2^{2s}\Gamma(s)}\lambda^{2s}\int\limits_0^\infty \frac 1{\tau^{1+s}}e^{-\frac {\lambda^2}{4\tau}}P_\tau u(x,t)\, \d\tau,
    \end{eqnarray}
    where ${P}_\tau u(x,t)$ was introduced in \eqref{eq3}, then
    \begin{equation}
\partial_\lambda(\lambda^{1-2s}\partial_\lambda{U})(\lambda,x,t) =\lambda^{1-2s} \mathcal{H}{U}(\lambda,x,t),\ {U}(0,x,t)  = u(x,t),\  (\lambda,x,t)\in \mathbb R_+\times\mathbb R^n\times \R,
\label{DN2+}
\end{equation}
and
\begin{align}\label{DN1+}
 -\lim\limits_{\lambda \rightarrow 0^+} \lambda^{1-2s} \partial_\lambda {U}(\lambda,x,t)=c_s\mathcal{H}^{s}u (x,t), \qquad (x,t)\in\R^n\times\R.
\end{align}
In particular, $\mathcal{H}^s$ can be realized as a Dirichlet to Neumann map using an extension problem, now for an associated (local) linear degenerate parabolic equation. As one consequence, local properties of solutions to $\mathcal{H}^{s}u=0$ in a domain $\Omega\times J\subset\mathbb R^n\times \mathbb R$ can be deduced by developing corresponding results for linear degenerate parabolic equations. Note that if ${U}$ is independent of $t$, then formally, the equation in \eqref{DN2+} coincides with the equation in the extension problem considered by Caffarelli and Silvestre \cite{caffarelli2007}.

\subsection{Continuous
time random walks (CTRWs)}
The kernel $K_{n,s}$ in \eqref{eq2} can also be given a  probabilistic interpretation  rooted in the context of continuous
time random walks (CTRWs).  To briefly discuss this, let $J_1,\ J_2,\ ...$ be nonnegative independent and identically distributed (i.i.d.) random variables that model the waiting
times between jumps of a particle that moves in $\mathbb R^n$. Set $T (0) := 0$ and let $T (m) := \sum_{j=1}^m J_j$ denote the time of the $m$-th jump. The particle jumps are given by i.i.d. random vectors $Y_1$, $Y_2$,... on  $\mathbb R^n$. Let $S(0) := 0$ and let $S(m) := \sum_{i=1}^m Y_i$ denote the position of the particle after the $m$-th
jump. For $t\geq 0$, let
$$N_t = \max\{m\geq 0 : T (m)\leq t\},$$
be the number of jumps up to time $t$. Define the stochastic process $\{X(t)\}_{t\geq 0}$
by
$$X(t) := S(N_t) =\sum_{i=1}^{N_t} Y_i.$$
Then $X(t)\in\mathbb R^n$ is the position of the particle at time $t$. $\{X(t)\}_{t\geq 0}$ is called a CTRW. Let $D$ be a stable subordinator with $E[e^{-\tau D}]=\exp(-\tau ^s)$ and assume that the (conditional) distribution for $A\in\mathbb R^n$, given $D=t$, is normal with mean 0 and variance $2t$. Then
$$E(e^{i\langle k,A\rangle})=E(E(e^{i\langle k,A\rangle}|D))=E(e^{-|k|^2D})=e^{-|k|^{2s}},$$
and hence  $A$ is stable with index $2s$. Let $a_m:=m^{-1/(2s)}$, $b_m:=m^{-1/s}$. Then by \cite{BMS04},
$$(a_mS(m),b_mT(m))\Rightarrow (A,D),$$
where $\Rightarrow $ denotes convergence in distribution. Furthermore, the limit $(A,D)$ has L\'{e}vy measure
$$\psi(\d x,\d t):=\frac 1{(4\pi t)^{n/2}}e^{-|x|^2/(4t)} \d x\bar\phi(\d t)=W_n(x,t,0,0)\d x\bar\phi(\d t),$$
where $\bar\phi$ is the L\'{e}vy measure for $D$. In particular, the Fourier-Laplace transform $(\mathcal{F}\mathcal{L})$ of the joint probability density $P_{(A,D)}$ can be calculated, and
$$(\mathcal{F}\mathcal{L})P_{(A,D)}(k,\tau)=e^{-\psi(k,\tau)},$$
where the log-characteristic function $\psi$ equals
$$\psi(k,\tau)=(\tau+|k|^2)^s.$$
This implies, as can be seen by inverting the Fourier-Laplace transform,  that the density  $f_\tau$ of $(A(\tau),D(\tau))$ equals  $K_{n,s}(x,-\tau)$ where $K_{n,s}$ is defined in \eqref{eq2}.

%%%%%%%%%%%%%%%%%%%%%%%%%%%%%%%%%%%%%%%%%%%%%%%

\subsection{Asymptotic mean value formulas} Given $(x,t)\in\mathbb R^{n+1}$, $r>0$, we let $C_r(x,t):=B(x,r)\times(t-r^2,t+r^2)$, where $B(x,r)$ is the standard open Euclidean ball in $\mathbb R^n$ centered at $x$ and with radius $r$. We also let $C_r^+(x,t):=B(x,r)\times(t,t+r^2)$ and $C_r^-(x,t):=B(x,r)\times(t-r^2,t)$.

Recall the following asymptotic mean value formula for the heat equation, the proof of which can be found in Section \ref{sec3}.

%%%%%%%%%%%%%%%%%%%%%%%%%%%%%%%%%%%%%%%%%%%%%%%
\begin{thm}\label{thm1-}
    Let $D\subset\mathbb R^{n+1}$ be an open set and let $u\in C^{2,1}(D)$.  Then $u$ solves the heat equation
    \begin{eqnarray}\label{solv1+}
        (\partial_t-\frac{1}{n+2} \Delta_x)u(x,t)=0\mbox { in }D,
    \end{eqnarray}
    if and only if the asymptotic mean value formula
    \begin{align}\label{meanvalue3+ha}
        u(x,t)&=\bariint_{C_\epsilon^-(x,t)} u(y,s)\, \d y\d s +{o}(\epsilon^{2}) \mbox{ as }\epsilon\to 0,
    \end{align}
    holds for every $(x,t)\in D$.
\end{thm}
%%%%%%%%%%%%%%%%%%%%%%%%%%%%%%%%%%%%%%%%%%%%%%%

Note that the factor $1/(n+2)$ in \eqref{solv1+} can be changed to $1$ at the expense of changing the domain of integration in time. We can give a probabilistic interpretation related to the formula in \eqref{meanvalue3+ha} as follows. We consider a random walk such that at each step,  a random waiting time $\tau$ and a change of the position $w$ of the token are drawn from the uniform distribution on $C_\epsilon^+(0,0)$. Heuristically, we can then use the dynamic programming principle to formulate an integral equation whose solution represents the expected value $v_\epsilon(x,t)$ of the game  starting at $(x,t)$, and 
\begin{align}\label{meanvalue3+hal}
v_\epsilon(x,t)&=\bariint_{C^+_\epsilon(x,t)} v_\epsilon(y,s)\, \d y\d s.
\end{align}
Note that, first, in this probabilistic interpretation, space and time are uncoupled, and second, compared to the  asymptotic mean value formula in Theorem \ref{thm1-}, integration in \eqref{meanvalue3+hal} is forward in time and the formula in \eqref{meanvalue3+hal} is exact, i.e., there is no $o(\epsilon^{2})$ term.

%%%%%%%%%%%%%%%%%%%%%%%%%%%%%%%%%%%%%%%%%%%%%%%
We next state  a version of Theorem \ref{thm1-}, but with $\mathcal{H}$ replaced by $\mathcal{H}^s$. To formulate the result we let
\begin{align}\label{eq1asymp1+}
    \mathcal{M}_\epsilon^s u(x,t)&:=\frac 1 2\kappa(n,s)\epsilon^{2s}\iint_{(\mathbb R^{n}\times\mathbb R_-)\setminus C_\epsilon^-(0,0)}{(u(x+w,t+\tau)+u(x-w,t+\tau))}K_{n,s}(w,-\tau)\, \d w\d \tau,
\end{align}
where
\begin{align}\label{eq1asymp2+}
    \kappa(n,s):=\biggl (\iint_{(\mathbb R^{n}\times\mathbb R_+)\setminus C_1^+(0,0)}K_{n,s}(w,\tau)\, \d w\d \tau\biggr )^{-1}.
\end{align}
A calculation shows that $\kappa(n,s)$ is finite.

We prove the following theorem. We refer to Section \ref{sec2} for the notion of  viscosity solutions for $\mathcal{H}^s$.

%%%%%%%%%%%%%%%%%%%%%%%%%%%%%%%%%%%%%%%%%%%%%%%

\begin{thm}\label{thm1a}
Let $D\subset\mathbb R^{n+1}$ be an open set and let $u\in C(D)\cap L^\infty(\mathbb R^{n+1})$.  Then $u$ is a viscosity solution to
\begin{eqnarray}\label{solv2+}
  \mathcal{H}^s u(x,t)=0\mbox { in }D,
\end{eqnarray}
if and only if the asymptotic mean value formula
\begin{align}\label{FHeat_meanvalue3+ha}
    u(x,t)&=\mathcal{M}_\epsilon^s u(x,t)+\mathcal O(\epsilon^{2}),\mbox{ as }\epsilon\to 0,
\end{align}
holds for every $(x,t)\in D$ in the viscosity sense.
\end{thm}

%%%%%%%%%%%%%%%%%%%%%%%%%%%%%%%%%%%%%%%%%%%%%%%
%Generalizations of Theorem \ref{thm1a} to anisotropic operators are discussed and formulated in the paper.
We can also give a probabilistic interpretation related to  Theorem \ref{thm1a} as follows. We consider a random walk such that at each step,  a random waiting time $\tau$ and a change of the position $w$ of the token are drawn from the distribution $K_{n,s}(w,\tau)$ conditioned on the event that $(w,\tau)\in (\mathbb R^{n}\times\mathbb R_+)\setminus C_\epsilon^+(0,0)$. Then, by dynamic programming principle, the expected value $v_\epsilon(x,t)$ of the game  starting at $(x,t)$ equals
\begin{align}\label{meanvalue3+hall}
v_\epsilon(x,t)&= \kappa(n,s)\epsilon^{2s}\iint_{(\mathbb R^{n}\times\mathbb R_+)\setminus C_\epsilon^+(0,0)}v_\epsilon(x+w,t+\tau)K_{n,s}(w,\tau)\, \d w\d \tau\notag\\
&=\frac 1 2\kappa(n,s)\epsilon^{2s}\iint_{(\mathbb R^{n}\times\mathbb R_+)\setminus C_\epsilon^+(0,0)}{(v_\epsilon(x+w,t+\tau)+v_\epsilon(x-w,t+\tau))}K_{n,s}(w,\tau)\, \d w\d \tau,
\end{align}
by the symmetry of the kernel with respect to $w$. Note that space and time are now coupled, and once again,  compared to the  asymptotic mean value formula in Theorem \ref{thm1a}, integration in \eqref{meanvalue3+hall} is forward in time and the formula in
\eqref{meanvalue3+hall} is exact, i.e. there is no $\mathcal O(\epsilon^{2})$ term.

\section{The nonlocal, nonlinear parabolic operators $\mathcal H^{s,\pm}_\infty$ and space-time tug-of-wars}\label{sec1}

%%%%%%%%%%%%%%%%%%%%%%%%%%%%%%%%%%%%%%%%%%%%%%%
Partially inspired  by \cite{bjor}, we here introduce, and this is the main innovation presented in this paper, a new class of nonlocal and nonlinear operators: the forward and backward infinity fractional heat operators $\mathcal H^{s,\pm}_\infty$. In this section, we let $S^n$ denote the unit sphere in $\mathbb R^n$.

\begin{defn}\label{IFLdefn}
For $s\in (0,1)$, $(x,t)\in\mathbb R^{n+1}$, the forward and backward infinity fractional heat operators $\mathcal H^{s,\pm}_\infty$ at a point $(x,t)$ are defined in the following way. Consider $u\in C^{2,1}(C_{2\delta}(x,t))\cap L^\infty(\mathbb R^{n+1})$ for some $\delta>0$.  If $\nabla_xu(x,t)\neq 0$, then
\begin{align}
\mathcal H^{s,\pm}_\infty u(x,t):=\frac{1}{2}\iint_{\mathbb R\times\mathbb R} [2u(x,t)-u(x+|\eta| v,t+\tau)-u(x-|\eta|v,t+\tau)]\,K_{1,s}(\eta,\pm\tau)\d\eta\d\tau
\notag
\end{align}
where $v=\nabla_xu(x,t)/|\nabla_xu(x,t)|\in S^{n}$. If $\nabla_xu(x,t)=0$, then
\begin{align}
\mathcal H^{s,\pm}_\infty u(x,t):&=\frac{1}{2}\sup_{y\in S^{n}}\left\{\iint_{\mathbb R\times\mathbb R} [u(x,t)-u(x+|\eta| y,t+\tau)]\,K_{1,s}(\eta,\pm\tau)\d\eta\d\tau\right\}\notag\\
&+\frac{1}{2}\inf_{z\in S^{n}}\left\{\iint_{\mathbb R\times\mathbb R} [u(x,t)-u(x-|\eta| z,t+\tau)]\,K_{1,s}(\eta,\pm\tau)\d\eta\d\tau\right\}.
\end{align}
\end{defn}

Given
$y\in S^{n}$, we also introduce
\begin{align}\label{eq1asymp1+haetto1}
\mathcal{M}_{\epsilon,\pm}^{s,y} u(x,t):= \kappa(1,s)\epsilon^{2s}\left\{\iint_{(\mathbb R\times\mathbb R)\setminus C_\epsilon(0,0)} u(x+|\eta|y,t+\tau)\,K_{1,s}(\eta,\pm\tau)\d\eta\d\tau\right\},
\end{align}
and given
$y\in S^{n}$ and $z\in S^{n}$,
\begin{align}\label{eq1asymp1+ha}
\mathcal{M}_{\epsilon,\pm}^{s,y,z} u(x,t):=\frac 1 2\mathcal{M}_{\epsilon,\pm}^{s,y} u(x,t)+\frac 1 2\mathcal{M}_{\epsilon,\pm}^{s,-z} u(x,t).
\end{align}

%%%%%%%%%%%%%%%%%%%%%%%%%%%%%%%%%%%%%%%%%%%%%%%
\begin{defn}\label{IFLdefn+}
For $s\in (0,1)$, $(x,t)\in\mathbb R^{n+1}$, the mean value  operators $\mathcal{M}_{\epsilon,\pm}^{s,\infty}$ at a point $(x,t)$ are defined in the following way. Consider $u\in C^{2,1}(C_{2\delta}(x,t))\cap L^\infty(\mathbb R^{n+1})$ for some $\delta>0$.  If $\nabla_x u(x,t)\neq 0$, then
\begin{align}
\mathcal{M}_{\epsilon,\pm}^{s,\infty}u(x,t):=\mathcal{M}_{\epsilon,\pm}^{s,v,v} u(x,t),\notag
\end{align}
where $v=\nabla_x u(x,t)/|\nabla_x u(x,t)|\in S^{n}$. If $\nabla_x u(x,t)=0$, then
\begin{align}
\mathcal{M}_{\epsilon,\pm}^{s,\infty} u(x,t):=\sup_{y\in S^{n}}\inf_{z\in S^{n}}\mathcal{M}_{\epsilon,\pm}^{s,y,z}  u(x,t).
\end{align}
\end{defn}

%%%%%%%%%%%%%%%%%%%%%%%%%%%%%%%%%%%%%%%%%%%%%%%
Using the notions introduced, we prove the following theorem. Note that Theorem \ref{thm1aba} is stated for the backwards infinity fractional heat operator in order to stay consistent with the previous theorems. The theorem, however, also holds for the forward operator. We refer to Section \ref{viscosity_formalism_ifho} for the notion of  viscosity solutions for $ \mathcal H^{s,\pm}_\infty$.
\begin{thm}\label{thm1aba}
Let $D\subset\mathbb R^{n+1}$ be an open set, let $u\in C(D)\cap L^\infty(\mathbb R^{n+1})$ and $s\in (0,1)$.  Then $u$ is a viscosity solution to
\begin{eqnarray}\label{solv3+}
  \mathcal H^{s,-}_\infty u(x,t)=0\mbox { in }D,
\end{eqnarray}
if and only if the asymptotic mean value formula
\begin{align}\label{IFHeat_meanvalue3+ha}
u(x,t)&=\mathcal{M}_{\epsilon,-}^{s,\infty} u(x,t)+\mathcal O(\epsilon^{2}),\mbox{ as }\epsilon\to 0,
\end{align}
holds for every $(x,t)\in D$ in the viscosity sense.
\end{thm}

%%%%%%%%%%%%%%%%%%%%%%%%%%%%%%%%%%%%%%%%%%%%%%%
We can give a probabilistic interpretation related to  Theorem \ref{thm1aba} in terms of  a space-time coupled tug-of-war game as follows. Two players play a zero-sum game and at each  step of the game, the players pick directions $y\in S^{n}$ and $z\in S^{n}$, respectively. Then a random waiting time $\tau$ and a one-dimensional position change $\eta$ are drawn from the distribution $K_{1,s}(\eta,\tau)$ conditioned on the event that $(\eta,\tau)\in (\mathbb R^{n}\times\mathbb R_+)\setminus C_\epsilon^+(0,0)$. At $t+\tau$, the token is moved in space as $x\to x+\eta y$ if $\eta>0$ and as $x\to x+\eta z$ if $\eta<0$.  By the  dynamic programming principle, the expected value $v_\epsilon(x,t)$ of the game starting at $(x,t)$ satisfies
\begin{align} \label{meanvalue3+halll}
v_\epsilon(x,t)&=\sup_{y\in S^{n}}\inf_{z\in S^{n}}\frac 1 2\bigl (\mathcal{M}_{\epsilon,+}^{s,y} v_\epsilon(x,t)+\mathcal{M}_{\epsilon,+}^{s,-z} v_\epsilon(x,t)\bigr).
\end{align}
Once again,  compared to the  asymptotic mean value formula in Theorem \ref{thm1aba}, integration in \eqref{meanvalue3+halll} is forward in time and the formula in
\eqref{meanvalue3+halll} is exact, i.e. there is no $\mathcal O(\epsilon^{2})$ term.

As we will show in Section \ref{heuristic_derivation}, if $s\in (0,1)$, $ u\in C^{2,1}(C_{2\delta}(x,t))\cap L^\infty(\mathbb R^{n+1})$, and $\nabla_x u(x,t)\neq 0$, then in the limit $\epsilon\to 0$, the optimization problem
\begin{align}\label{Opt}
&\sup_{y\in S^{n}}\inf_{z\in S^{n}}\left\{\iint_{\mathbb R\times\mathbb R_+} [2 u(x,t)- u(x+|\eta| y,t+\tau)- u(x-|\eta| z,t+\tau)]\,K_{1,s}(\eta,\tau)\d\eta\d\tau\right\},
\end{align}
is solved by $y=z=v$, where $v=\nabla_x u(x,t)/|\nabla_x u(x,t)|\in S^{n}$. This explains the introduction of the operators $\mathcal H^{s,\pm}_\infty$ and formally (if
$\nabla_xv_\epsilon(x,t)\neq 0$),
\begin{align*}
v_\epsilon(x,t)=\mathcal{M}_{\epsilon,+}^{s,\infty} v_\epsilon(x,t).
\end{align*}

%%%%%%%%%%%%%%%%%%%%%%%%%%%%%%%%%%%%%%%%%%%%%%%
The purpose of this paper is to prove the asymptotic mean value formula for $\mathcal H^s$ stated in Theorem \ref{thm1a}, introduce the operators  $\mathcal H^{s,\pm}_\infty $, prove the asymptotic mean value formula for $\mathcal H^{s,\pm}_\infty$ stated in Theorem \ref{thm1aba}, and to prove asymptotic properties for both  $\mathcal H^s$ and $\mathcal H^{s,\pm}_\infty $. This paper opens up a number of interesting problems related to  properties of the operators $\mathcal H^{s,\pm}_\infty $ and concerning solutions to $\mathcal H^{s,\pm}_\infty u=0$.  Our analysis is in no way  final; instead, this paper should be seen as the first step towards a theory similar to \cite{graddep,bjor} but modelled on $(\partial_t- \Delta_x )^{s}$ instead of the fractional Laplacian $-(- \Delta_x )^{s}$. To our knowledge, the operators $\mathcal H^{s,\pm}_\infty $ have previously not been discussed in literature, but we note that when applied to functions which are independent of time, then  the operators formally coincide with $-\Delta^s_\infty$, where $\Delta^s_\infty$ is the infinity fractional Laplacian introduced in \cite{bjor}. We note that in the stationary case, versions of Theorem \ref{thm1a} are proved in \cite{ClaMarNA}.

%%%%%%%%%%%%%%%%%%%%%%%%%%%%%%%%%%%%%%%%%%%%%%%
%%%%%%%%%%%%%%%%%%%%%%%%%%%%%%%%%%%%%%%%%%%%%%%
\section{Viscosity formalism for $\mathcal{H}^s$}\label{sec2}

Let $D\subset\mathbb R^{n+1}$ be an open set. We denote by $\mbox{LSC}(D)$ the set of lower semi-continuous functions, i.e., all functions $f: D \to \R$ such that for all points $(\hat x,\hat t)\in D$ and for any sequence $\{(x_m, t_m)\}_m$, $(x_m, t_m)\in D$,   $(x_m,t_m) \to (\hat x,\hat t)$ as $m \to \infty$ in $D$, we have
\begin{equation*}
\liminf _{m \to \infty} f(x_m,t_m) \geq f(\hat x,\hat t).
\end{equation*}
We denote by
$\mbox{USC}(D)$ the set of upper semi-continuous functions, i.e., all functions $f: D \to \R$ such that for all points $(\hat x,\hat t)\in D$ and for any sequence $\{(x_m, t_m)\}_m$, $(x_m,  t_m)\in D$,   $(x_m,t_m) \to (\hat x,\hat t)$ as $n \to \infty$ in $D$, we have
\begin{equation*}
\limsup_{m \to \infty}f(x_m,t_m) \leq f(\hat x,\hat t).
\end{equation*}
Note that a function $f\in \mbox{USC}(D)$  if and only if $-f\in \mbox{LSC}(D)$. Also, $f$ is  continuous on $D$, $f\in C(D)$, if and only if
$f\in \mbox{USC}(D)\cap \mbox{LSC}(D)$.

Consider the equation
\begin{eqnarray}\label{solv1}
 \mathcal{H}^s u=0\mbox{ in }D.
\end{eqnarray}
Given $u\in L^\infty(\mathbb R^{n+1})$, an open set $U\subset D$, and  $\phi \in C^{2,1}(\overline U)$, we introduce the auxiliary (test) function
\syslab[v_{U,\phi,u} :=]{&\phi , \quad \mbox{ in }  U	\\
							&u, \quad \mbox{ in } \mathbb R^{n+1}\setminus U.
							\label{vvv}
							}

%%%%%%%%%%%%%%%%%%%%%%%%%%%%%%%%%%%%%%%%%%%%%%%
\begin{defn}\label{Vissolagain}
Let  $u\in L^\infty(\mathbb R^{n+1})$ and let $D\subset\mathbb R^{n+1}$ be an open set. $u\in \mbox{LSC}(D)$ is a viscosity supersolution to the equation in \eqref{solv1} in $D$ if for every $(\hat x,\hat t)\in D$, any neighborhood $U=U(\hat x,\hat t)\subset D$ and any $\phi \in C^{2,1}(\overline U)$ such that
\begin{eqnarray*}
(1)&& u(\hat x,\hat t)=\phi(\hat x,\hat t),\notag\\
(2)&& u(x,t)>\phi(x,t)\mbox{ for all }(x,t)\in U\setminus\{\hat x,\hat t\},
\end{eqnarray*}
it holds that
\begin{eqnarray*}
\mathcal{H}^s v(\hat x,\hat t)\geq 0,\mbox{ where } v:=v_{U,\phi,u}.
\end{eqnarray*}
$u\in \mbox{USC}(D)$ is a viscosity subsolution to the equation in \eqref{solv1} in $D$ if for every $(\hat x,\hat t)\in D$, any neighborhood $U=U(\hat x,\hat t)\subset D$ and any $\phi \in C^{2,1}(\overline U)$ such that
\begin{eqnarray*}
(1)&& u(\hat x,\hat t)=\phi(\hat x,\hat t),\notag\\
(2)&& u(x,t)<\phi(x,t)\mbox{ for all }(x,t)\in U\setminus\{\hat x,\hat t\},
\end{eqnarray*}
then,
\begin{eqnarray*}
\mathcal{H}^s v(\hat x,\hat t)\leq 0, \mbox{ where } v:=v_{U,\phi,u}.
\end{eqnarray*}
A function $u\in C(D)$ is said to be a viscosity solution to \eqref{solv1} in $D$ if it is both a viscosity supersolution and a viscosity subsolution to \eqref{solv1} in $D$.
\end{defn}

%%%%%%%%%%%%%%%%%%%%%%%%%%%%%%%%%%%%%%%%%%%%%%%
\begin{defn}\label{meanvalueviscosity}
Let $D\subset\mathbb R^{n+1}$ be an open set and let $u\in C(D)\cap L^\infty(\mathbb R^{n+1})$. We say that $u$ satisfies the asymptotic mean value formula
\begin{align}\label{meanvalue3+}
u(x,t)&=\mathcal{M}_\epsilon^s u(x,t)+\mathcal O(\epsilon^{2}),\mbox{ as }\epsilon\to 0,
\end{align}
in the viscosity sense at $(x,t)\in D$ if the following holds. For every $\phi$ as in Definition \ref{Vissolagain}, and touching $u$ from below, we have
\begin{align}\label{meanvalue3++}
v(x,t)&\geq\mathcal{M}_\epsilon^s v(x,t)+\mathcal O(\epsilon^{2}),\mbox{ as }\epsilon\to 0,
\end{align}
where $v$ is defined in \eqref{vvv}, and  for every $\phi$ as in Definition \ref{Vissolagain}, and touching $u$ from above, we have
\begin{align}\label{meanvalue3++gg}
v(x,t)&\leq\mathcal{M}_\epsilon^s v(x,t)+\mathcal O(\epsilon^{2}),\mbox{ as }\epsilon\to 0,
\end{align}
where $v$ is defined in \eqref{vvv}.
\end{defn}

%%%%%%%%%%%%%%%%%%%%%%%%%%%%%%%%%%%%%%%%%%%%%%%
%%%%%%%%%%%%%%%%%%%%%%%%%%%%%%%%%%%%%%%%%%%%%%%
\section{Proof of Theorem \ref{thm1-} and Theorem \ref{thm1a}}\label{sec3}

\subsection{Proof of Theorem \ref{thm1-}} Let $D\subset\mathbb R^{n+1}$ be an open set and let $u\in C^{2,1}(D)$. Let $(x,t)\in D$ and $(y,s)\in C_\epsilon^-(x,t)$. Using Taylor's formula at $(x,t)$, we have
\begin{align}
u(y,s)&= u(x,t)+\nabla_xu(x,t)\cdot(y-x)\notag\\
&+\frac 1 2 \langle \nabla_x^2u(x,t)(y-x), (y-x)\rangle+\partial_tu(x,t)(s-t)+ {o}(\epsilon^2),
\end{align}
as $\epsilon\to 0$. Taking the average,
\begin{align*}
\bariint_{C_\epsilon^-(x,t)}u(y,s) \, \d y\d s=& u(x,t)+\bariint_{C_\epsilon^-(x,t)}\nabla_xu(x,t)\cdot(y-x)\notag \, \d y\d s\\
&+\bariint_{C_\epsilon^-(x,t)}\frac 1 2 \langle \nabla_x^2u(x,t)(y-x), (y-x)\rangle \, \d y\d s\notag\\
&+\bariint_{C_\epsilon^-(x,t)}\partial_tu(x,t)(s-t)\, \d y\d s+ {o}(\epsilon^2).
\end{align*}
Note that the time-slices of $ C_\epsilon^-(x,t)$ are symmetric with respect to $x$. Therefore, by symmetry, the contribution from the term $\nabla_xu(x,t)\cdot(y-x)$ to the average is zero, and therefore, the first integral term on the right-hand side vanishes. Again, by symmetry, we can write the second integral as
\begin{align*}
\bariint_{C_\epsilon^-(x,t)}\frac 1 2 \langle \nabla_x^2u(x,t)(y-x), (y-x)\rangle \, \d y\d s&= \frac 1 2 \bariint_{C_\epsilon^-(x,t)} \sum_{i,j}\partial_{x_ix_j} u(x,t) (y_i-x_i)(y_j-x_j) \, \d y\d s\\
&= \frac 1 2 \bariint_{C_\epsilon^-(x,t)} \sum_{i}\partial_{x_ix_i} u(x,t) (y_i-x_i)^2 \, \d y\d s\\
&= \frac 1 {2n} \Delta_x u(x,t)   \bariint_{C_\epsilon^-(x,t)} (y-x)^2 \, \d y\d s.
\end{align*}
A calculation then shows that
\begin{align*}
\bariint_{C_\epsilon^-(x,t)}\frac 1 2 \langle \nabla_x^2u(x,t)(y-x), (y-x)\rangle \, \d y\d s= \frac{\epsilon^2}{2(n+2)}\Delta_x u(x,t).
\end{align*}
 Similarly,
\begin{align*}
\bariint_{C_\epsilon^-(x,t)}\partial_tu(x,t)(s-t)\, \d y\d s&=-\frac{\epsilon^2}{2} \partial_tu(x,t).
\end{align*}
Hence,
\begin{align*}\label{meanvalue3+haproof}
\bariint_{C_\epsilon^-(x,t)} u(y,s)\, \d y\d s=u(x,t)-\frac{\epsilon^2}{2}(\partial_t-\frac{1}{(n+2)}\Delta_x)u(x,t))+{o}(\epsilon^{2}),\mbox{ as }\epsilon\to 0.
\end{align*}
This implies that the asymptotic mean value formula in Theorem \ref{thm1-} holds if and only if $u$ solves the heat equation with diffusion coefficient $(n+2)^{-1}$,
\[(\partial_t-\frac{1}{n+2} \Delta_x)u(x,t)=0.\]
%%%%%%%%%%%%%%%%%%%%%%%%%%%%%%%%%%%%%%%%%%%%%%%

%%%%%%%%%%%%%%%%%%%%%%%%%%%%%%%%%%%%%%%%%%%%%%%
\subsection{Proof of Theorem \ref{thm1a} for smooth functions}
We first establish a version of Theorem \ref{thm1a} for smooth functions.
\begin{lem}\label{FHeat_lemma_smooth}
	Let $(\hat x,\hat t) \in \mathbb{R}^{n+1}$ and consider $u\in C^{2,1}(C_{2\delta}(\hat x,\hat t))\cap L^\infty(\mathbb R^{n+1})$. Then 	
%	for any $r>0$
	\bgs{ u(x,t)= \mathcal{M}_\epsilon^s u(x,t) +
		  \kappa(n,s) \epsilon^{2s} \mathcal{H}^s u(x,t)	  +\mathcal O(\epsilon^{2})
	}
	for all $(x,t)\in C_{\delta}(\hat x,\hat t)$ as $\epsilon\to 0$.
\end{lem}

%%%%%%%%%%%%%%%%%%%%%%%%%%%%%%%%%%%%%%%%%%%%%%%
\begin{proof}
Fix $(x,t)\in C_{\delta}(\hat x,\hat t)$ and consider $0<\epsilon<\delta$. First, we note that
\begin{align*}
(\kappa(n,s))^{-1}: &= \iint_{\mathbb R^{n+1}\setminus C_1^-(0,0)}K_{n,s}(w,-\tau)\, \d w\d \tau\\
&= \iint_{\mathbb R^{n+1}\setminus C_{\epsilon}^-(0,0)}K_{n,s}(\epsilon w,-{\epsilon^{-2}\tau}) \epsilon^{-(n+2)}\, \d w\d \tau\\
&= \epsilon^{-(n+2)} \iint_{\mathbb R^{n+1}\setminus C_{\epsilon}^-(0,0)}K_{n,s}(w,\tau) \epsilon^{2(1+s)+n}\, \d w\d \tau\\
&=\epsilon^{2s} \iint_{\mathbb R^{n+1}\setminus C_{\epsilon}^-(0,0)}K_{n,s}(w,\tau)\, \d w\d \tau.
\end{align*}
We have
\begin{align}\label{Hs_smooth_u-M}
	&{u(x,t)- \mathcal{M}_\epsilon^s u(x,t)}\notag\\
	=&\, \kappa(n,s)\epsilon^{2s}\iint_{\mathbb R^{n+1}\setminus C_\epsilon^-(0,0)} u(x,t) K_{n,s}(w,-\tau)\, \d w\d \tau- \mathcal{M}_\epsilon^s u(x,t)\notag\\
=& \frac{1}{2}\kappa(n,s)\epsilon^{2s}\iint_{\mathbb R^{n+1}\setminus C_\epsilon^-(0,0)} {(2u(x,t)-u(x+w,t+\tau)-u(x-w,t+\tau))}K_{n,s}(w,-\tau)\, \d w\d \tau.
\end{align}
By a change of variables, and introducing
$$\delta_\epsilon(u,(x,t),(w,\tau)):=(2u(x,t)-u(x+\epsilon w,t+\epsilon^2\tau)-u(x- \epsilon w,t+\epsilon^2\tau)),$$
we have
\begin{align}\label{Hs_smooth_u-M_delta}
	{u(x,t)- \mathcal{M}_\epsilon^s u(x,t)}&=\frac{1}{2}\kappa(n,s)\epsilon^{2s}\iint_{\mathbb R^{n+1}\setminus C_{1}^-(0,0)} {\delta_\epsilon(u,(x,t),(w,\tau))}K_{n,s}(w,-\tau) \epsilon^{-(2+2s+n)}\epsilon^{2+n}\, \d w\d \tau\notag\\
	&=\frac{1}{2}\kappa(n,s)\iint_{\mathbb R^{n+1}\setminus C_{1}^-(0,0)} {\delta_\epsilon(u,(x,t),(w,\tau))}K_{n,s}(w,-\tau) \, \d w\d \tau.
\end{align}
Hence,
\begin{align}
\label{Hs_smooth_u-M_H-I}
	\frac {2u(x,t)- 2\mathcal{M}_\epsilon^s u(x,t)}{\kappa(n,s)}&= \iint_{\mathbb R^{n+1}\setminus C_{1}^-(0,0)} {\delta_\epsilon(u,(x,t),(w,\tau))}K_{n,s}(w,-\tau) \, \d w\d \tau \notag\\
	&= \iint_{\mathbb R^{n+1}} {\delta_\epsilon(u,(x,t),(w,\tau))}K_{n,s}(w,-\tau) \, \d w\d \tau- I_\epsilon(u,(x,t)) \notag\\
	&=2\epsilon^{2s}\mathcal{H}^s u(x,t)- I_\epsilon(u,(x,t)),
\end{align}
where
\begin{align}
\label{I}
I_\epsilon(u,(x,t)):=\iint_{C_1^-(0,0)} {\delta_\epsilon(u,(x,t),(w,\tau))}K_{n,s}(w,-\tau)\, \d w\d \tau.
\end{align}
Since $(\hat x+\epsilon w, \hat t+\epsilon^2\tau) \in C_{2\delta}(\hat x,\hat t)$, we can use Taylor's formula and deduce that
\[|\delta_\epsilon(u,(x,t),(w,\tau))|\leq c\epsilon^2(|w|^2+|\tau|)\]
for all $(w,\tau)\in C_1(0,0)$. Hence,
\begin{align}\label{first}
|I_\epsilon(u,(x,t))|\leq c\epsilon^2\iint_{C_1^-(0,0)} (|w|^2+|\tau|)K_{n,s}(w,-\tau)\, \d w\d \tau\leq \tilde{c}\epsilon^2.
\end{align}
Put together,
\[
	2u(x,t)- 2\mathcal{M}_\epsilon^s u(x,t)=2\kappa(n,s)\epsilon^{2s}\mathcal{H}^s u(x,t)- \kappa(n,s)I_\epsilon(u,(x,t)),
\]
which implies
\begin{align}\label{Hs_smooth_conc}
u(x,t)=\mathcal{M}_\epsilon^s u(x,t)+\kappa(n,s)\epsilon^{2s}\mathcal{H}^s u(x,t)+\mathcal O(\epsilon^{2}), 
\end{align}
proving the lemma.
\end{proof}

%%%%%%%%%%%%%%%%%%%%%%%%%%%%%%%%%%%%%%%%%%%%%%%
%%%%%%%%%%%%%%%%%%%%%%%%%%%%%%%%%%%%%%%%%%%%%%%
\subsection{Proof of Theorem \ref{thm1a}}
Let $(\hat x,\hat t)\in D$ be arbitrary and let $\delta>0$ be such that $U:=C_{2\delta}(\hat x,\hat t)$ satisfies $\overline {U}\subset D$. Let $\phi \in C^{2,1}(\overline{U})$ be such that
\begin{eqnarray*}
(1)&& u(\hat x,\hat t)=\phi(\hat x,\hat t),\notag\\
(2)&& u(x,t)>\phi(x,t)\mbox{ for all }(x,t)\in U\setminus\{\hat x, \hat t\}.
\end{eqnarray*}
We let $v:=v_{U,\phi,u}$  be defined as in \eqref{vvv}, hence $v\in C^{2,1}(U)\cap L^\infty(\mathbb R^{n+1})$.  By Lemma \ref{FHeat_lemma_smooth}, we have that
		\begin{eqnarray}\label{asaf}
v(\hat x,\hat t)= \mathcal{M}_\epsilon^sv(\hat x,\hat t) +  \kappa(n,s) \epsilon^{2s} \mathcal{H}^sv(\hat x,\hat t) + \mathcal O(\epsilon^{2}),\mbox{ as }\epsilon\to 0.
\end{eqnarray}

Assume that $u$ is a supersolution to $\mathcal{H}^s u(\hat x,\hat t)=0$ in the viscosity sense. Then
\[ \mathcal{H}^sv(\hat x,\hat t) \geq 0\]
and this implies that
\[v(\hat x,\hat t)\geq \mathcal{M}_\epsilon^sv(\hat x,\hat t)  + \mathcal O(\epsilon^{2}),\mbox{ as }\epsilon\to 0.
\]
Similarly, if $u$ is a subsolution to $\mathcal{H}^s u(\hat x,\hat t)=0$ in the viscosity sense, and if $\phi \in C^{2,1}(\overline{U})$  touches $u$ from above at $(\hat x,\hat t)$, then
\[v(\hat x,\hat t)\leq \mathcal{M}_\epsilon^sv(\hat x,\hat t)  + \mathcal O(\epsilon^{2}),\mbox{ as }\epsilon\to 0.
\]
In particular, if  $u$ is a viscosity solution to the equation in \eqref{solv2+}, then $u$ satisfies the asymptotic mean value formula
\begin{align*}
    u(x,t)&=\mathcal{M}_\epsilon^s u(x,t)+\mathcal O(\epsilon^{2}),\mbox{ as }\epsilon\to 0,
\end{align*}
for every $(x,t)\in D$ in the viscosity sense.

Next, assume instead that $u$ satisfies (\ref{FHeat_meanvalue3+ha}) in the viscosity sense. Then, assuming that $\phi$ touches $u$ from below at $(\hat x,\hat t)$ we have
\[v(\hat x,\hat t)\geq \mathcal{M}_\epsilon^s v(\hat x,\hat t)+ \mathcal O (\epsilon^{2}).\]
Therefore, by dividing by  $\kappa(n,s)\epsilon^{2s}$ in \eqref{asaf} and sending $\epsilon\to 0$, we deduce that
\[\mathcal{H}^s v(\hat x,\hat t)\geq 0.\]
This proves that $u$ is a supersolution to $\mathcal{H}^s u(x,t)=0$ at $(\hat x,\hat t)$ in the viscosity sense. An analogous argument shows that $u$ is also a subsolution to $\mathcal{H}^s u(x,t)=0$ at $(\hat x,\hat t)$ in the viscosity sense. I.e., if $u$ satisfies the asymptotic mean value formula in (\ref{FHeat_meanvalue3+ha}) in the viscosity sense, then
$u$ is a solution to $\mathcal{H}^s u(x,t)=0$ in the viscosity sense.

%%%%%%%%%%%%%%%%%%%%%%%%%%%%%%%%%%%%%%%%%%%%%%%
%%%%%%%%%%%%%%%%%%%%%%%%%%%%%%%%%%%%%%%%%%%%%%%

%here 
%%%%%%%%%%%%%%%%%%%%%%%%%%%%%%%%%%%%%%%%%%%%%%%
%%%%%%%%%%%%%%%%%%%%%%%%%%%%%%%%%%%%%%%%%%%%%%%
\section{Derivation of $\mathcal{H}^{s,\pm}_\infty$ and viscosity solutions}\label{sec4}
In this section we introduce $\mathcal{H}^{s,\pm}_\infty$ and viscosity solutions. For simplicity we only consider the operator $\mathcal{H}^{s,-}_\infty$ as the operator 
$\mathcal{H}^{s,+}_\infty$ can be handled by analogy.

%%%%%%%%%%%%%%%%%%%%%%%%%%%%%%%%%%%%%%%%%%%%%%%
\subsection{Heuristic derivation of $\mathcal{H}^{s,-}_\infty$}\label{heuristic_derivation}
Consider the  optimization problem in \eqref{Opt}. Assume that $u\in C^{2,1}(C_{2\delta}(x,t))\cap L^\infty(\mathbb R^{n+1})$ for some $\delta>0$. We write
\begin{align}\label{Opt1+}
&2 u(x,t)- u(x+|\eta| y,t+\tau)- u(x-|\eta| z,t+\tau)\notag\\
&=-( u(x+|\eta| y,t+\tau)+ u(x-|\eta| z,t+\tau)-|\eta|\nabla_x u(x,t)\cdot (y-z)-2 u(x,t))\notag\\
&\ -|\eta|\nabla_x u(x,t)\cdot (y-z).
\end{align}
Assume that $(\eta,\tau)\in (\mathbb R\times\mathbb R_-)\cap C_\delta(0,0)$. Then
\begin{align}\label{Opt2+}
&| u(x+|\eta| y,t+\tau)+ u(x-|\eta| z,t+\tau)-|\eta|\nabla_x u(x,t)\cdot (y-z)-2 u(x,t)|\notag\\
&\leq c( u,\delta)(|\tau|+|\eta|^2)
\end{align}
for some $c( u,\delta)$. It is easily checked that
\begin{align}\label{Opta}
\iint_{ (\mathbb R\times\mathbb R_-)\cap C_\delta(0,0)} c( u,\delta)(|\tau|+|\eta|^2)\,K_{1,s}(\eta,-\tau)\d\eta\d\tau  = \tilde{c} \int_0^{\delta^2} \tau^{-s}d\tau<\infty
\end{align}
for some $\tilde{c}$, whenever $s\in (0,1)$. We have $(\mathbb R\times\mathbb R_-)\setminus C_\delta(0,0)=D_1\cup D_2$ where
\[D_1:= \{(\eta,\tau): \tau\leq -\delta^2\},\enskip D_2:= \{(\eta,\tau): |\eta|>\delta, -\delta^2<\tau<0\}.\]
 If $(\eta,\tau)\in (\mathbb R\times\mathbb R_-)\setminus C_\delta(0,0)$, then
\begin{align}\label{Opt3+}
&|( u(x+|\eta| y,t+\tau)+ u(x-|\eta| z,t+\tau)-|\eta|\nabla_x u(x,t)\cdot (y-z)-2 u(x,t))|\leq c\max\{\delta,|\eta|\},
\end{align}
for some  $c$ as $ u$ is globally bounded. Furthermore,
\begin{align}\label{Optab+}
\iint_{D_2} \max\{\delta,|\eta|\}\,K_{1,s}(\eta,-\tau)\d\eta\d\tau&=\int_{-\delta^2}^0 \int_{|\eta|\geq \delta}|\eta|K_{1,s}(\eta,-\tau)\d\eta\d\tau\notag\\
&=-\frac{1}{\Gamma(-s)}\int_{0}^{\delta^2} \int_{|\eta|\geq \delta}|\eta|\frac{e^{-\eta^2/(4\tau)}}{(4\pi\tau)^{1/2}}\frac 1{\tau^{1+s}}\d\eta\d\tau\notag\\
&=\frac{2}{\Gamma(-s)}\int_{0}^{\delta^2} \int_{|\eta|\geq \delta}\frac{\partial_\eta(e^{-\eta^2/(4\tau)})}{(4\pi\tau)^{1/2}}\frac 1{\tau^{s}}\d\eta\d\tau\notag\\
&=-\frac{4}{\Gamma(-s)}\int_{0}^{\delta^2}\frac{e^{-\delta^2/(4\tau)}}{(4\pi\tau)^{1/2}}\frac 1{\tau^{s}}\d\tau<\infty.
\end{align}
Similarly,
\begin{align}\label{Optab}
\iint_{D_1} \max\{\delta,|\eta|\}\,K_{1,s}(\eta,-\tau)\d\eta\d\tau=&\delta\int_{\delta^2}^{\infty}\int_0^\delta K_{1,s}(\eta,\tau)\d\eta\d\tau\notag\\
&+\int_{\delta^2}^{\infty}\int_{\delta}^{\infty} |\eta|K_{1,s}(\eta,\tau)\d\eta\d\tau<\infty.
\end{align}
Combining these estimates, we see that
\begin{align}\label{Optgg}
&\iint_{\mathbb R\times\mathbb R_-} [2 u(x,t)- u(x+|\eta| y,t+\tau)- u(x-|\eta| z,t+\tau)]\,K_{1,s}(\eta,-\tau)\d\eta\d\tau\notag\\
&=I+\nabla_x u(x,t)\cdot (z-y)\iint_{\mathbb R\times\mathbb R_-} |\eta|\,K_{1,s}(\eta,-\tau)\d\eta\d\tau,
\end{align}
where $I$ is finite for every $s\in (0,1)$. However,
\begin{align}\label{Optggg}
\iint_{\mathbb R\times\mathbb R_-} |\eta|\,K_{1,s}(\eta,-\tau)\d\eta\d\tau&=\iint_{\mathbb R\times\mathbb R_+} |\eta|\frac{e^{-\eta^2/(4\tau)}}{(4\pi\tau)^{1/2}}\frac 1{\tau^{1+s}}\d\eta\d\tau\notag\\
&=\hat c\int_{\mathbb{R}} \eta^{-2s}\d\eta
\end{align}
for some $\hat c$, and the last integral is divergent if $s\in(0,1)$ and then equal to $+\infty$. Hence, two situations emerge. First, if $\nabla_x u(x,t)=0$, then the divergent term disappears and this has to be treated separately. Note, that if $\nabla_x u(x,t)=0$, then the two expressions stated in the optimization problem in \eqref{Opt} are both bounded from above and below. This is a again a consequence of  regularity of $ u$ at $(x,t)$ and of the global boundedness of $ u$. Second, if $\nabla_x u(x,t)\neq 0$, then we let $v=\nabla_x u(x,t)/|\nabla_x u(x,t)|$ and we note that
$$\nabla_x u(x,t)\cdot(z-\nabla_x u(x,t))\geq 0$$
for all $z\in S^n$. Hence, as $y=v$ is a candidate in the optimization problem, we are compelled to also choose $z=v$. This leads us to introduce  the fractional parabolic infinity operators as stated in Definition \ref{IFLdefn}.

\begin{remark}
In \cite{bjor}, the authors consider $(2s-1)$-H\"{o}lder functions, and define the infinity fractional Laplacian only for $s\in(\frac{1}{2},1)$. In this paper, we consider globally bounded and locally $C^{2,1}$ functions. In this smaller family, we can define the parabolic version of infinity fractional Laplacian for all $s\in(0,1)$.
\end{remark}
%here

%%%%%%%%%%%%%%%%%%%%%%%%%%%%%%%%%%%%%%%%%%%%%%%
\subsection{Viscosity formalism for $\mathcal{H}_\infty^{s,-}$} \label{viscosity_formalism_ifho}
Let $D\subset\mathbb R^{n+1}$ be an open set, $(x,t)\in D$ and consider $u\in C^{2,1}(C_{2\delta}(x,t))\cap L^\infty(\mathbb R^{n+1})$.  

We say that $ u$ is a supersolution to
$\mathcal{H}_\infty^{s,-}u=0$ at $(x,t)$  if there exists, for any $\xi>0$ small, $y_\xi\in S^{n}$ such that
\begin{align}
\inf_{z\in S^{n}}\left\{\iint_{\mathbb R\times\mathbb R} [2 u(x,t)- u(x+|\eta| y_\xi,t+\tau)- u(x-|\eta| z,t+\tau)]\,K_{1,s}(\eta,-\tau)\d\eta\d\tau\right\}\geq \xi.\label{infop}
\end{align}
If $\nabla_x  u(x,t)=0$, one can deduce, again using Taylor's formula as above, that the above integral is finite for any choice of $y_\xi$ and $z$. As $S^{n}$ is compact, there is
a subsequence $\xi_i\rightarrow 0$ and $y_{\xi_i}\rightarrow y_0$ such that, in the limit,
\begin{align}
\inf_{z\in S^{n}}\left\{\iint_{\mathbb R\times\mathbb R} [2 u(x,t)- u(x+|\eta| y_0,t+\tau)- u(x-|\eta| z,t+\tau)]\,K_{1,s}(\eta,-\tau)\d\eta\d\tau\right\}\geq 0.\notag
\end{align}
Hence, in the case $\nabla_x  u(x,t)=0$, we say that $u$ is a supersolution to $\mathcal{H}_\infty^{s,-}u=0$ at $(x,t)$ if there is a $y_0$ such that the above inequality holds. On the other hand,  if $\nabla_x  u(x,t)\neq 0$, we say that $u$ is a supersolution to $\mathcal{H}_\infty^{s,-}u=0$ at $(x,t)$ if
\begin{align*}
\iint_{\mathbb R\times\mathbb R} [2 u(x,t)- u(x-|\eta| v,t+\tau)- u(x+|\eta| v,t+\tau)]\,K_{1,s}(\eta,-\tau)\d\eta\d\tau\geq 0
\end{align*}
where $v={\nabla_x  u(x,t)}/{|\nabla_x  u(x,t)|}$. If $ u$ is a supersolution at $(x,t)$ we write $\mathcal{H}_\infty^{s,-}u(x,t)\geq 0$.

We say that $ u$ is a subsolution $\mathcal{H}_\infty^{s,-}u=0$ at $(x,t)$  if there exists, for any $\xi>0$ small, $y_\xi\in S^{n}$ such that
\begin{align}
\sup_{y\in S^{n}}\left\{\iint_{\mathbb R\times\mathbb R} [2 u(x,t)- u(x+|\eta| y,t+\tau)- u(x-|\eta| z_\xi,t+\tau)]\,K_{1,s}(\eta,-\tau)\d\eta\d\tau\right\}\leq \xi.\label{supop}
\end{align}
Again,  $\nabla_x  u(x,t)=0$, then the above integral is finite for any choice of $y$ and $z_\xi$. As $S^{n}$ is compact, there is
a subsequence $\xi_i\rightarrow 0$ and $z_{\xi_i}\rightarrow z_0$ such that, in the limit,
\begin{align}
\sup_{y\in S^{n}}\left\{\iint_{\mathbb R\times\mathbb R} [2 u(x,t)- u(x+|\eta| y,t+\tau)- u(x-|\eta| z_0,t+\tau)]\,K_{1,s}(\eta,-\tau)\d\eta\d\tau\right\}\leq 0.\notag
\end{align}
So, in the case $\nabla_x  u(x,t)=0$, we say that $ u$ is a subsolution to $\mathcal{H}_\infty^{s,-}u=0$ at $(x,t)$ if there is a $z_0$ such that the above inequality holds. On the other hand, by the heuristic argument in Section \ref{heuristic_derivation} above, we are lead to say that, if $\nabla_x  u(x,t)\neq 0$, $ u$ is a subsolution to $\mathcal{H}_\infty^{s,-}u=0$ at $(x,t)$ if
\begin{align*}
\iint_{\mathbb R\times\mathbb R} [2 u(x,t)- u(x-|\eta| v,t+\tau)- u(x+|\eta| v,t+\tau)]\,K_{1,s}(\eta,-\tau)\d\eta\d\tau\leq 0
\end{align*}
where $v={\nabla_x  u(x,t)}/{|\nabla_x  u(x,t)|}$.
If $ u$ is a subsolution to $\mathcal{H}_\infty^{s,-}u=0$ at $(x,t)$, we write $\mathcal{H}_\infty^s u(x,t)\leq 0$.

We consider the equation
\begin{eqnarray}\label{solv1more}
  \mathcal{H}_\infty^{s,-}  u=0\mbox{ in }D.
\end{eqnarray}

%%%%%%%%%%%%%%%%%%%%%%%%%%%%%%%%%%%%%%%%%%%%%%%
\begin{defn}\label{Vissolagainmore} Let  $u\in L^\infty(\mathbb R^{n+1})$ and let $D\subset\mathbb R^{n+1}$ be an open set. $u\in \mbox{LSC}(D)$ is
a viscosity supersolution to the equation in \eqref{solv1more} in $D$ if for every $(\hat x,\hat t)\in D$, any neighborhood $U=U(\hat x,\hat t)\subset D$ and any $\phi \in C^{2,1}(\overline U)$ such that
\begin{eqnarray*}
(1)&& u(\hat x,\hat t)=\phi(\hat x,\hat t),\notag\\
(2)&& u(x,t)>\phi(x,t)\mbox{ for all }(x,t)\in U\setminus\{\hat x,\hat t\},
\end{eqnarray*}
it holds that
\begin{eqnarray*}
 \mathcal{H}_\infty^{s,-}  v(\hat x,\hat t)\geq 0,\mbox{ where } v:=v_{U,\phi,u}.
\end{eqnarray*}
$u\in \mbox{USC}(D)$ is
a viscosity subrsolution to the equation in \eqref{solv1more} in $D$ if for every $(\hat x,\hat t)\in D$, any neighborhood $U=U(\hat x,\hat t)\subset D$ and any $\phi \in C^{2,1}(\overline U)$ such that
\begin{eqnarray*}
(1)&& u(\hat x,\hat t)=\phi(\hat x,\hat t),\notag\\
(2)&& u(x,t)<\phi(x,t)\mbox{ for all }(x,t)\in U\setminus\{\hat x,\hat t\},
\end{eqnarray*}
then,
\begin{eqnarray*}
 \mathcal{H}_\infty^{s,-}  v(\hat x,\hat t)\leq 0, \mbox{ where } v:=v_{U,\phi,u}.
\end{eqnarray*}
A function $u\in C(D)\cap L^\infty(\mathbb R^{n+1})$ is said to be a viscosity solution to
\eqref{solv1more} in $D$ if it is both a viscosity supersolution and a viscosity subsolution to \eqref{solv1more} in $D$.
\end{defn}

Recall that if $v$ is defined as in \eqref{vvv}, then $\mathcal{M}_{\epsilon,\pm}^{s,\infty} v(x,t)$ is defined as in Definition \ref{IFLdefn+}.

%%%%%%%%%%%%%%%%%%%%%%%%%%%%%%%%%%%%%%%%%%%%%%%
\begin{defn}\label{meanvalueviscosityinfty}  Let $D\subset\mathbb R^{n+1}$ be an open set and let $u\in C(D)\cap L^\infty(\mathbb R^{n+1})$. We say that $u$ satisfies the asymptotic mean value formula
\begin{align}\label{meanvalue3+infty}
u(x,t)&=\mathcal{M}_{\epsilon,-}^{s,\infty} u(x,t)+\mathcal O(\epsilon^{2}),\mbox{ as }\epsilon\to 0,
\end{align}
in the viscosity sense at $(x,t)\in D$ if the following holds. For every $\phi$ as in Definition \ref{Vissolagainmore}, and touching $u$ from below, we have
\begin{align}\label{meanvalue3++infty}
v(x,t)&\geq\mathcal{M}_{\epsilon,-}^{s,\infty} v(x,t)+\mathcal O(\epsilon^{2}),\mbox{ as }\epsilon\to 0,
\end{align}
where $v$ is defined in \eqref{vvv}, and for every $\phi$ as in Definition \ref{Vissolagainmore}, and touching $u$ from above, we have
\begin{align}\label{meanvalue3++gginfty}
v(x,t)&\leq\mathcal{M}_{\epsilon,-}^{s,\infty} v(x,t)+\mathcal O(\epsilon^{2}),\mbox{ as }\epsilon\to 0,
\end{align}
where $v$ is defined in \eqref{vvv}.
\end{defn}

%%%%%%%%%%%%%%%%%%%%%%%%%%%%%%%%%%%%%%%%%%%%%%%

%%%%%%%%%%%%%%%%%%%%%%%%%%%%%%%%%%%%%%%%%%%%%%%
%%%%%%%%%%%%%%%%%%%%%%%%%%%%%%%%%%%%%%%%%%%%%%%
\section{Proof of Theorem \ref{thm1aba}}\label{sec5}

We first establish  a version of Theorem \ref{thm1aba} for smooth functions.

\begin{lem}\label{asymp}
Let $(\hat x,\hat t)\in \mathbb R^{n+1}$ and consider $u\in C^{2,1}(C_{2\delta}(\hat x,\hat t))\cap L^\infty(\mathbb R^{n+1})$. Then 	
\bgs{
    u(x,t)= \mathcal{M}_{\epsilon,-}^{s,\infty} u(x,t) + \kappa(1,s) \epsilon^{2s} \mathcal{H}^{s,-}_\infty u(x,t) + \mathcal O(\epsilon^{2})
    }
for all $(x,t)\in C_{\delta}(\hat x,\hat t)$ as $\epsilon\to 0$.
\end{lem}

\begin{proof} Fix $(x,t)\in C_{\delta}(\hat x,\hat t)$, $y,z\in S^n$ and consider $0<\epsilon<\delta$.  As with the proof of Lemma \ref{FHeat_lemma_smooth} we have by the definition of $\kappa(1,s)$ in \eqref{eq1asymp2+},
\begin{align}\label{IFH_smooth_u-My}
	&{u(x,t)- \mathcal{M}_{\epsilon,-}^{s,y} u(x,t)}\notag\\
	=&\, \kappa(1,s)\epsilon^{2s}\iint_{\mathbb R^{2}\setminus C_\epsilon(0,0)} u(x,t) K_{1,s}(\eta,-\tau)\, \d\eta\d \tau- \mathcal{M}_{\epsilon,-}^{s,y} u(x,t)\notag\\
=& \kappa(1,s)\epsilon^{2s}\iint_{\mathbb R^{2}\setminus C_\epsilon(0,0)} {(u(x,t)-u(x+|\eta|y,t+\tau))}K_{1,s}(\eta,-\tau)\, \d \eta\d \tau,
\end{align}
and
\begin{align}\label{IFH_smooth_u-Mz}
	&{u(x,t)- \mathcal{M}_{\epsilon,-}^{s,-z} u(x,t)}\notag\\
	=&\, \kappa(1,s)\epsilon^{2s}\iint_{\mathbb R^2\setminus C_\epsilon(0,0)} u(x,t) K_{1,s}(\eta,-\tau)\, \d\eta\d \tau- \mathcal{M}_{\epsilon,-}^{s,-z} u(x,t)\notag\\
=& \kappa(1,s)\epsilon^{2s}\iint_{\mathbb R^{2}\setminus C_\epsilon(0,0)} {(u(x,t)-u(x-|\eta|z,t+\tau))}K_{1,s}(\eta,-\tau)\, \d \eta\d \tau.
\end{align}
We note in both cases, i.e., the cases $\nabla_xu(x,t)\neq 0$ and $\nabla_xu(x,t)= 0$, that we can write
\begin{align}\label{IFH_smooth_u-M_exp}
u(x,t)- \mathcal{M}_{\epsilon,-}^{s,y} u(x,t)&=\kappa(1,s)\epsilon^{2s}\iint_{\mathbb R^{2}} {(u(x,t)-u(x+|\eta|y,t+\tau))}K_{1,s}(\eta,-\tau)\, \d \eta\d \tau\notag\\
&-\kappa(1,s)\iint_{C_1(0,0)}(u(x,t)-u(x+\epsilon|\eta|y,t+\epsilon^2\tau))K_{1,s}(\eta,-\tau)\, \d \eta\d \tau,
\end{align}
as both integrals are finite. Furthermore,
\begin{align}\label{IFH_smooth_u-Myz}
	&{u(x,t)- \mathcal{M}_{\epsilon,-}^{s,y,z} u(x,t)}\notag\\
	=&\, \kappa(1,s)\epsilon^{2s}\iint_{\mathbb R^{2}\setminus C_\epsilon(0,0)} u(x,t) K_{1,s}(\eta,-\tau)\, \d\eta\d \tau- \mathcal{M}_{\epsilon,-}^{s,y,z} u(x,t)\notag\\
=& \frac{1}{2}\kappa(1,s)\epsilon^{2s}\iint_{\mathbb R^{2}\setminus C_\epsilon(0,0)} {(2u(x,t)-u(x+|\eta|y,t+\tau)-u(x-|\eta|z,t+\tau))}K_{1,s}(\eta,-\tau)\, \d \eta\d \tau.
\end{align}
By a change of variables, and introducing $$\delta_\epsilon^{y,z}(u,(x,t),(\eta,\tau)):=2u(x,t)-u(x+\epsilon|\eta|y,t+\epsilon^2\tau)-u(x-\epsilon|\eta|z,t+\epsilon^2\tau),$$ we have
\begin{align}\label{IFH_smooth_u-Myz_I_J}
	u(x,t)-\mathcal{M}_{\epsilon,-}^{s,y,z} u(x,t) = \frac{1}{2}\kappa(1,s)\epsilon^{2s}J_-^{s,y,z}(u,(x,t)) -\frac{1}{2}\kappa(1,s) I_{\epsilon,-}^{s,y,z} (u,(x,t)),
\end{align}
where
\begin{align*}
J_-^{s,y,z}(u,(x,t))&:=\iint_{\mathbb R^{2}} {(2u(x,t)-u(x+|\eta|y,t+\tau)-u(x-|\eta|z,t+\tau))}K_{1,s}(\eta,-\tau)\, \d \eta\d \tau,\\
I_{\epsilon,-}^{s,y,z} (u,(x,t))&:=\iint_{C_1(0,0)} {\delta_\epsilon^{y,z}(u,(x,t),(\eta,\tau))}K_{1,s}(\eta,-\tau)\, \d\eta\d \tau.
\end{align*}

Assume that $\nabla_xu(x,t)\neq 0$ and let  $y=z=v$, where $v=\nabla_x u(x,t)/|\nabla_x u(x,t)|\in S^{n}$. Then, using the regularity of $u$ in $C_{2\delta}(x,t)$, and a Taylor expansion, we deduce that
\begin{align}\label{IFH_smooth_abs-delta}
 |\delta_\epsilon^{y,z}(u,(x,t),(\eta,\tau))|&=|(2u(x,t)-u(x-\epsilon|\eta|v,t+\epsilon^2\tau)-u(x+\epsilon|\eta|v,t+\epsilon^2\tau))|\notag\\
 &\leq c\epsilon^2(|\eta v|^2+|\tau|),
\end{align}
and hence
\begin{align}\label{IFH_smooth_abs-I}
|I_{\epsilon,-}^{s,y,z} (u,(x,t))|\leq  c\epsilon^2\iint_{C_1(0,0)} (|\eta v|^2+|\tau|)K_{1,s}(\eta,-\tau)\, \d \eta\d \tau\leq \tilde{c}\epsilon^2.
\end{align}
Recalling from Definition \ref{IFLdefn+} that $\mathcal{M}_{\epsilon,-}^{s,\infty}u(x,t):=\mathcal{M}_{\epsilon,-}^{s,v,v}u(x,t)$ and noting that $\frac{1}{2}J_-^{s,y,z}(u,(x,t)) = \mathcal{H}_{\infty}^{s,-}$, we have proved, using the above estimates, that in the case $\nabla_xu(x,t)\neq 0$,
\begin{align}\label{IFH_smooth_grad-not0_end}
	u(x,t) = \mathcal{M}_{\epsilon,-}^{s,\infty}u(x,t) + \kappa(1,s)\epsilon^{2s}\mathcal{H}_{\infty}^{s,-}u(x,t) + \mathcal{O}(\epsilon^2),
\end{align}
as $\epsilon\to 0$.

Assume that $\nabla_xu(x,t)=0$. We have from (\ref{IFH_smooth_u-M_exp}) that
\begin{align}\label{IFH_smooth_u-My_L}
	&{u(x,t)- \mathcal{M}_{\epsilon,-}^{s,y} u(x,t)}\notag\\
	=&\, \kappa(1,s)\epsilon^{2s}\iint_{\mathbb R^{2}} {(u(x,t)-u(x+|\eta|y,t+\tau))}K_{1,s}(\eta,-\tau)\, \d \eta\d \tau - \kappa(1,s)L_{\epsilon,-}^{s,y} (u,(x,t)),
\end{align}
where
\begin{align*}
L_{\epsilon,-}^{s,y} (u,(x,t))&:=\iint_{C_1(0,0)}(u(x,t)-u(x+\epsilon|\eta|y,t+\epsilon^2\tau))K_{1,s}(\eta,-\tau)\, \d \eta\d \tau.
\end{align*}
Now,
$$|u(x,t)-u(x+\epsilon|\eta|y,t+\epsilon^2\tau)|\leq c\epsilon^2(|\tau|)$$
when $(\eta,\tau)\in C_1(0,0)$, hence
\begin{align}\label{IFH_smooth_abs-L}
|L_{\epsilon,-}^{s,y} (u,(x,t))|\leq  c\epsilon^2\iint_{C_1(0,0)} |\tau|K_{1,s}(\eta,-\tau)\, \d \eta\d \tau\leq \tilde{c}\epsilon^2,
\end{align}
and so,
\begin{align}\label{IFH_smooth_u-My_bigO}
&u(x,t)- \mathcal{M}_{\epsilon,-}^{s,y} u(x,t)\notag\\
=&\,\kappa(1,s)\epsilon^{2s}\iint_{\mathbb R^{2}} {(u(x,t)-u(x+|\eta|y,t+\tau))}K_{1,s}(\eta,-\tau)\, \d \eta\d \tau+\mathcal{O}(\epsilon^2)
\end{align}
for all $y\in S^{n}$. Similarly,
\begin{align}\label{IFH_smooth_u-infM}
&u(x,t)- \mathcal{M}_{\epsilon,-}^{s,-z} u(x,t)\notag\\
=&\,\kappa(1,s)\epsilon^{2s}\iint_{\mathbb R^{2}} {(u(x,t)-u(x-|\eta|z,t+\tau))}K_{1,s}(\eta,-\tau)\, \d \eta\d \tau+\mathcal{O}(\epsilon^2).
\end{align}
Consequently, recalling (\ref{eq1asymp1+ha}),
\begin{align*}
u(x,t)=& \mathcal{M}_{\epsilon,-}^{s,y,z} u(x,t)+\frac{1}{2}\kappa(1,s)\epsilon^{2s}\iint_{\mathbb R^{2}} {(u(x,t)-u(x+|\eta|y,t+\tau))}K_{1,s}(\eta,-\tau)\, \d \eta\d \tau\\
&+
\frac{1}{2}\kappa(1,s)\epsilon^{2s}\iint_{\mathbb R^{2}} {(u(x,t)-u(x-|\eta|z,t+\tau))}K_{1,s}(\eta,-\tau)\, \d \eta\d \tau+\mathcal{O}(\epsilon^2).
\end{align*}
First taking the infimum over all $z\in S^{n}$, and then taking the supremum over all $y\in S^{n}$, we have that
\begin{align}\label{IFH_smooth_u-supinfM}
	u(x,t) &= \sup_{y\in S^{n}}\inf_{z\in S^{n}}\mathcal{M}_{\epsilon,-}^{s,y,z} u(x,t)\notag\\
	&+\, \frac{1}{2}\kappa(1,s)\epsilon^{2s}\sup_{y\in S^{n}}\biggl\{\iint_{\mathbb R^{2}} {(u(x,t)-u(x+|\eta|y,t+\tau))}K_{1,s}(\eta,-\tau)\, \d \eta\d \tau\biggr \} \notag\\
	&+ \frac{1}{2}\kappa(1,s)\epsilon^{2s}\inf_{z\in S^{n}}\biggl\{\iint_{\mathbb R^{2}} {(u(x,t)-u(x-|\eta|z,t+\tau))}K_{1,s}(\eta,-\tau)\, \d \eta\d \tau\biggr \}+\mathcal{O}(\epsilon^2).
\end{align}
Finally, from Definitions \ref{IFLdefn} and \ref{IFLdefn+},
\begin{align}\label{IFH_smooth_grad-0_end}
u(x,t)=\mathcal{M}_{\epsilon,-}^{s,\infty}u(x,t) + \kappa(1,s)\epsilon^{2s}\mathcal{H}_\infty^{s,-}u(x,t) + \mathcal O(\epsilon^{2}),
\end{align}
which completes the proof of the lemma.
\end{proof}

%%%%%%%%%%%%%%%%%%%%%%%%%%%%%%%%%%%%%%%%%%%%%%%
%%%%%%%%%%%%%%%%%%%%%%%%%%%%%%%%%%%%%%%%%%%%%%%
\subsection{Proof of Theorem \ref{thm1aba}: the general case}
Let $(\hat x, \hat t)\in D$ be arbitrary and let $\delta>0$ be such that $U:=C_{2\delta}(\hat x, \hat t)$ satisfies $\overline {U}\subset D$. Let $\phi \in C^{2,1}(\overline{U})$ be such that
\begin{eqnarray*}
(1)&& u(\hat x, \hat t)=\phi(\hat x, \hat t),\notag\\
(2)&& u(x,t)>\phi(x,t)\mbox{ for all }(x,t)\in U\setminus\{\hat x, \hat t\}.
\end{eqnarray*}
We let $v:=v_{U,\phi,u}$  be defined as in \eqref{vvv}, hence $v\in C^{2,1}(U)\cap L^\infty(\mathbb R^{n+1})$.  By Lemma \ref{asymp}, we have that
\begin{eqnarray}\label{IFHeat_gen_bylemma}
v(\hat x,  t)= \mathcal{M}_{\epsilon,-}^{s,\infty}v( x,  t) +  \kappa(1,s) \epsilon^{2s} \mathcal{H}_\infty^{s,-}v( x,  t) + \mathcal O(\epsilon^{2}),\mbox{ as }\epsilon\to 0,
\end{eqnarray}
$(x, t)\in C_{\delta}( x, t)$ as $\epsilon\to 0$.

Assume that $u$ is a supersolution to $\mathcal{H}_\infty^{s,-}u(\hat x, \hat t)=0$ in the viscosity sense. Then
\[ \mathcal{H}_\infty^{s,-}v(\hat x, \hat t) \geq 0\]
implies that
\[v(\hat x, \hat t)\geq \mathcal{M}_{\epsilon,-}^{s,\infty}v(\hat x, \hat t)  + \mathcal O(\epsilon^{2}),\mbox{ as }\epsilon\to 0.
\]
Similarly, if $u$ is a subsolution to $\mathcal{H}_\infty^{s,-}u(\hat x, \hat t)=0$ in the viscosity sense, then
\[v(\hat x, \hat t)\leq \mathcal{M}_{\epsilon,-}^{s,\infty}v(\hat x, \hat t)  + \mathcal O(\epsilon^{2}),\mbox{ as }\epsilon\to 0.
\]
Since the choices of $(\hat x, \hat t)$ and $\phi$ are arbitrary, we have that $u$ satisfies (\ref{IFHeat_meanvalue3+ha}) in $D$ in the viscosity sense.

Assume that $u$ satisfies (\ref{IFHeat_meanvalue3+ha}) in the viscosity sense. Then, as
\[v(\hat x, \hat t)\geq \mathcal{M}_{\epsilon,-}^{s,\infty}v(\hat x, \hat t)+ \mathcal O (\epsilon^{2}),\]
it follows, by dividing by  $\kappa(1,s)\epsilon^{2s}$ in \eqref{IFHeat_gen_bylemma} and sending $\epsilon\to 0$, that
\[\mathcal{H}_\infty^{s,-}v(\hat x, \hat t)\geq 0.\]
Therefore, $u$ is a supersolution to $\mathcal{H}_\infty^{s,-}u(\hat x, \hat t)=0$ in the viscosity sense. Similar arguments show that $u$ is also a subsolution to $\mathcal{H}_\infty^{s,-}u(\hat x, \hat t)=0$ in the viscosity sense, which completes the proof.

%%%%%%%%%%%%%%%%%%%%%%%%%%%%%%%%%%%%%%%%%%%%%%%
%%%%%%%%%%%%%%%%%%%%%%%%%%%%%%%%%%%%%%%%%%%%%%%

\section{Asymptotic Properties as $s\nearrow1$}\label{sec6}

In this section, we provide some asymptotic properties on the operators $\mathcal{H}^s$ and $\mathcal{H}_\infty^{s,\pm}$. In particular, we study their limit behavior as $s$ approaches the upper bound $1$. Indeed, by sending $s\nearrow1$, we obtain the local counterpart of the operators under study.

\subsection{Limit of $\mathcal{H}^s$ as $s\nearrow 1$}

\begin{lem}\label{asymp1} Let $u\in C^{2,1}(D)\cap L^\infty(\mathbb R^{n+1})$. Then for all $(x,t)\in D$
\begin{align*}
     \lim_{s\nearrow 1}\mathcal{H}^s u(x,t)&=(\partial_t-\Delta_x) u(x,t).
          \end{align*}
\end{lem}
\begin{proof} Given $u\in C^{2,1}(D)\cap L^\infty(\mathbb R^{n+1})$ and
$$\delta(u,(x,t),(w,\tau)):=2u(x,t)-u(x+w,t+\tau)-u(x-w,t+\tau),$$
we recall from (\ref{eq1asymp+intr}) the representation
\begin{align*}
     \mathcal{H}^s u(x,t)&=\frac 1 2 \iint_{\mathbb R^{n+1}} \delta(u,(x,t),(w,\tau))K_{n,s}(w,-\tau)\, \d w\d \tau \\
     &= \frac 1 2 \int_{\mathbb R^{n}\times\mathbb R_-}\delta(u,(x,t),(w,\tau))K_{n,s}(w,-\tau)\, \d w\d \tau=:\frac{1}{2}I_s(x,t).
     \end{align*}
 We fix $\eps\in(0,1)$, which we will take arbitrarily small in the sequel. Consider the following sub-domains:
 \begin{align*}
 D_1&:= B(0,\epsilon)\times \mathbb{R}_-,\\
 D_2&:=\mathbb{R}^n \times (-\epsilon^2,0) ,\\
 D_3&:=\mathbb{R}^n \times \mathbb{R}_- \setminus \{D_1\cup D_2\}.
 \end{align*}
 Note that $D_1\cap D_2 = C_{\epsilon}^-(0,0)$. We can then write
  \begin{align*}
 I_s(x,t) =&\iint_{D_1} \; + \iint_{D_2} \; +\iint_{D_3} \; -\iint_{D_1\cap D_2}.
 \end{align*}

 First, since $u$ is globally bounded, we have in $D_3$, for some constant $c>0$,
 \begin{align*}
 \iint_{D_3} \delta(u,(x,t),(w,\tau))K_{n,s}(w,-\tau)\, \d w\d \tau&\leq c \iint_{D_3} K_{n,s}(w,-\tau) \, \d w\d \tau\\
 &\leq -\frac{c \epsilon^{-2s} }{\Gamma(-s) (4\pi)^{\frac{n}{2}}}  \int_{||w||>1}\left(\int_0^{\infty} \tau^{-1-s-\frac{n}{2}} e^{-\frac{w^2}{4\tau}}\d\tau \right)\d w\\
% &= -\frac{c \epsilon^{-2s} 4^s}{\Gamma(-s) (\pi)^{\frac{n}{2}}}   \int_{||w||>1} w^{-n-2s} \Gamma(\frac{n}{2}+s) \d w\\
 &=-\frac{c \epsilon^{-2s} 4^s \Gamma(\frac{n}{2}+s)|S^{n-1}|}{\Gamma(-s) (\pi)^{\frac{n}{2}}}  \int_1^{\infty} r^{-2s-1} dr\\
 %&=-\frac{cn \epsilon^{-2s} 4^s \Gamma(\frac{n}{2}+s)}{\Gamma(-s)\Gamma(\frac{n}{2}+1) }  \int_1^{\infty} r^{-2s-1} dr\\
 %&=-\frac{cn \epsilon^{-2s} 4^{s-\frac{1}{2}} \Gamma(\frac{n}{2}+s)}{s\Gamma(-s)\Gamma(\frac{n}{2}+1) }\\
 &=\frac{cn \epsilon^{-2s} 4^{s-\frac{1}{2}} (1-s)\Gamma(\frac{n}{2}+s)}{\Gamma(2-s)\Gamma(\frac{n}{2}+1) }.
 \end{align*}
Sending $s\nearrow 1$, we see that this integral tends to $0$.
In $D_1$, we have
\begin{align*}
\delta(u,(x,t),(w,\tau))K_{n,s}(w,-\tau) =&(2u(x,t)-u(x+w,t)-u(x-w,t)+u(x+w,t)+u(x-w,t)\\
&-u(x+w,t+\tau)-u(x-w,t+\tau))K_{n,s}(w,-\tau)\\
=&\left(\langle\nabla_x^2u(x,t)w,w\rangle+ \mathcal{O}(\epsilon^2)\right)K_{n,s}(w,-\tau)+(u(x+w,t)\\
&+u(x-w,t)-u(x+w,t+\tau)-u(x-w,t+\tau))K_{n,s}(w,-\tau).
\end{align*}
Similarly, in $D_2$, we have
\begin{align*}
\delta(u,(x,t),(w,\tau))K_{n,s}(w,-\tau) = &\left(-2\partial_tu(x,t)\tau+ \mathcal{O}(\epsilon^2)\right)K_{n,s}(w,-\tau)\\
&+\left(2u(x,t)-u(x+w,t)-u(x-w,t)\right)K_{n,s}(w,-\tau).
\end{align*}
Therefore,
\begin{align*}
&  \iint_{D_1} \; + \iint_{D_2}  \;-\iint_{D_1\cap D_2} \ \\
=&  \iint_{D_1} \left(\langle\nabla_x^2u(x,t)w,w\rangle+ \mathcal{O}(\epsilon^2)\right)K_{n,s}(w,-\tau)\, \d w\d \tau\notag\\
& +\iint_{D_2}\left(-2\partial_tu(x,t)\tau+ \mathcal{O}(\epsilon^2)\right)K_{n,s}(w,-\tau)\, \d w\d \tau\\
&+ \iint_{D_1\setminus D_2} \left(u(x+w,t)+u(x-w,t)-u(x+w,t+\tau)-u(x-w,t+\tau)\right)K_{n,s}(w,-\tau)\, \d w\d \tau\\
&+\iint_{D_2\setminus D_1} \left(2u(x,t)-u(x+w,t)-u(x-w,t)\right)K_{n,s}(w,-\tau)\, \d w\d \tau.
\end{align*}
By the same argument as the integral in $D_3$, the last two integrals will tend to $0$ as $s\nearrow 1$ due to boundedness of $u$ outside of $D_1\cap D_2$. The first integral can be written as
\begin{align*}
&\iint_{D_1} \left(\langle\nabla_x^2u(x,t)w,w\rangle+ \mathcal{O}(\epsilon^2)\right)K_{n,s}(w,-\tau)\, \d w\d \tau\\
=&\iint_{B_{\epsilon}(0)\times \mathbb{R^-}}\langle\nabla_x^2u(x,t)w,w\rangle \frac{1}{(4\pi)^{\frac{n}{2}}\Gamma(-s)} \frac{1}{(-\tau)^{1+s+\frac{n}{2}}} e^{\frac{w^2}{4\tau}}\, \d w\d \tau+ \mathcal{O}(\epsilon^2)\\
=&\frac{\Delta_x u(x,t)}{n(4\pi)^{\frac{n}{2}}\Gamma(-s)} \iint_{B_{\epsilon}(0)\times \mathbb{R^+}} w^2\tau^{-(1+s+\frac{n}{2})} e^{-\frac{w^2}{4\tau}}\, \d w\d \tau+ \mathcal{O}(\epsilon^2)\\
%=&\frac{\Delta_x u(x,t)}{n(4\pi)^{\frac{n}{2}}\Gamma(-s)} \epsilon^{2-2s}\iint_{B_{1}(0)\times \mathbb{R^+}} w^2\tau^{-(1+s+\frac{n}{2})} e^{-\frac{w^2}{4\tau}}\, \d w\d \tau+ \mathcal{O}(\epsilon^2)\\
%=& \frac{\Delta_x u(x,t)}{n(4\pi)^{\frac{n}{2}}\Gamma(-s)} \epsilon^{2-2s}\int_{B_{1}(0)} w^2\left(\int_{\mathbb{R^+}} \tau^{-(1+s+\frac{n}{2})} e^{-\frac{w^2}{4\tau}}\d \tau\right) \d w+ \mathcal{O}(\epsilon^2)\\
%=&\frac{4^{s+\frac{n}{2}}\Delta_x u(x,t)}{n(4\pi)^{\frac{n}{2}}\Gamma(-s)} \epsilon^{2-2s}\int_{B_{1}(0)} w^{2-n-2s} \Gamma(\frac{n}{2}+s) \d w+ \mathcal{O}(\epsilon^2)\\
%=&\frac{4^{s}\Delta_x u(x,t)\Gamma(\frac{n}{2}+s)}{n(\pi)^{\frac{n}{2}}\Gamma(-s)} \epsilon^{2-2s}\int_{S^{n-1}} \int_0^1 r^{1-2s}\d r \d\sigma+ \mathcal{O}(\epsilon^2)\\
=&\frac{4^{s}\Delta_x u(x,t)\Gamma(\frac{n}{2}+s)}{n(\pi)^{\frac{n}{2}}\Gamma(-s)} \epsilon^{2-2s}  \frac{2\pi^{\frac{n}{2}}}{\Gamma(\frac{n}{2})} \frac{1}{2(1-s)} + \mathcal{O}(\epsilon^2)\\
=&-\frac{ 4^s\Delta_x u(x,t)s\Gamma(\frac{n}{2}+s)}{n\Gamma(2-s)\Gamma(\frac{n}{2}) }\epsilon^{2-2s} + \mathcal{O}(\epsilon^2).
\end{align*}
 Letting  $s\nearrow1$, the above expression goes to $-2\Delta_x u(x,t) $. Similarly, the second integral can be written as
   \begin{align*}
   &\iint_{D_2}\left(-2\partial_tu(x,t)\tau+ \mathcal{O}(\epsilon^2)\right)K_{n,s}(w,-\tau)\, \d w\d \tau\\
=& -2\partial_tu(x,t)\iint_{D_2}\tau K_{n,s}(w,-\tau)\, \d w\d \tau+ \mathcal{O}(\epsilon^2)\\
%=&-2\partial_tu(x,t) \epsilon^{2-2s} \iint_{\mathbb{R}^n \times (0,1)}\tau \frac{1}{(4\pi)^{\frac{n}{2}}\Gamma(-s)} \frac{1}{\tau^{1+s+\frac{n}{2}}} e^{-\frac{w^2}{4\tau}}\, \d w\d \tau+ \mathcal{O}(\epsilon^2)\\
%=& -\frac{2\partial_tu(x,t)}{\Gamma(-s)} \epsilon^{2-2s} \int_0^1 \frac{1}{\tau^{s+\frac{n}{2}}(4\pi)^{\frac{n}{2}} }\left( \int_{\mathbb{R}^n} e^{-\frac{w^2}{4\tau}}\, \d w \right) \d\tau + \mathcal{O}(\epsilon^2)\\
=&-\frac{2\partial_tu(x,t) }{\Gamma(-s)} \epsilon^{2-2s} \int_0^1\tau^{-s}\d\tau+ \mathcal{O}(\epsilon^2)\\
=& \frac{2s \epsilon^{2-2s}}{\Gamma(2-s)}\partial_tu(x,t) + \mathcal{O}(\epsilon^2),
   \end{align*}
    which tends to $2\partial_tu(x,t)$ if we first let  $s\nearrow1$ then $\epsilon \to 0$. Therefore, $I _s(x,t)\to (\partial_t - \Delta_x )u(x,t)$ as $s\nearrow1$ and the conclusion follows.
\end{proof}

\subsection{Limit of $\mathcal{H}_\infty^{s,\pm}$ as $s\nearrow 1$}

\begin{lem}\label{asymp2} Let $u\in C^{2,1}(D)\cap L^\infty(\mathbb R^{n+1})$. Consider  $(x,t)\in D$ and assume that $\nabla_xu(x,t)\neq 0$.
\begin{align*}
     \lim_{s\nearrow 1}\mathcal{H}_\infty^{s,\pm}u(x,t)&=\partial_tu(x,t)\pm\Delta_{\infty,x}^N u(x,t).
          \end{align*}
\end{lem}
\begin{proof} If $\nabla_xu(x,t)\neq 0$, we have that
\begin{align}
\mathcal H^{s,\pm}_\infty u(x,t):=\frac{1}{2}\iint_{\mathbb R\times\mathbb R} [2u(x,t)-u(x+|\eta| v,t+\tau)-u(x-|\eta|v,t+\tau)]\,K_{1,s}(\eta,\pm\tau)\d\eta\d\tau
\notag
\end{align}
where $v=\nabla_xu(x,t)/|\nabla_xu(x,t)|\in S^{n}$. Following the same calculations as in Lemma \ref{asymp1} and using
$$\delta^v(u,(x,t),(\eta,\tau)):=2u(x,t)-u(x+|\eta|v,t+\tau)-u(x-|\eta|v,t+\tau),$$
we can derive
\begin{align*}
   \lim_{s\nearrow 1}\mathcal{H}_\infty^{s,-}u(x,t)&= \lim_{s\nearrow 1}\frac{1}{2}\iint_{\mathbb{R}\times \mathbb{R_-}}\delta^v(u,(x,t),(\eta,\tau))K_{1,s}(w,-\tau) d\eta d\tau\\
   &= \lim_{\epsilon \to 0} \lim_{s\nearrow 1}  \left(\frac{2s }{\Gamma(2-s)} \partial_tu(x,t) \epsilon^{2-2s}- \frac{s4^s \Gamma(s+\frac{1}{2}) }{\Gamma(2-s) \Gamma(\frac{1}{2})}\langle\nabla_x^2u(x,t)v,v\rangle  + \mathcal{O}(\epsilon^2) \right)\\
   &= \partial_tu(x,t)-\langle\nabla_x^2u(x,t)\frac{\nabla_x u(x,t)}{|\nabla_x u(x,t)|},\frac{\nabla_x u(x,t)}{|\nabla_x u(x,t)|}\rangle\\
  &= \partial_tu(x,t)-\Delta_{\infty,x}^N u(x,t).
\end{align*}
The asymptotics for $\mathcal{H}_\infty^{s,+}$ follows similarly.
\end{proof}

\section{Future research and open problems}\label{sec7}
In this paper, we have in particular introduced the forward and backward infinity fractional heat operators $\mathcal H^{s,\pm}_\infty$,    we have proved asymptotic mean value formulas for viscosity solutions to $\mathcal H^{s,\pm}_\infty u=0$, and we have established the asymptotic limits of $\mathcal H^{s,\pm}_\infty$ as $s\nearrow 1$. As previously mentioned, our analysis is just the beginning of what we hope are further developments on this topic. Below, we state a few  open problems for future research. \\

\noindent
{\bf Problem 1}: In \cite{bjor}, the authors study the existence, uniqueness, and regularity of both the Dirichlet problem and obstacle problem for the infinity fractional Laplacian. It is an interesting problem to attempt a similar analysis for the time-dependent operators $\mathcal H^{s,\pm}_\infty$.\\

\noindent{\bf Problem 2}: In \cite{her}, the authors study the stochastic representation of a certain class of space-time coupled operators. Based on their results, one could attempt to study the connection between $\mathcal H^{s,\pm}_\infty$ and CTRWs as well as corresponding Cauchy problems.\\

\noindent{\bf Problem 3}: In this paper, we have defined solution to $\mathcal H^{s,\pm}_\infty u=0$, for all $s\in(0,1)$, in the viscosity sense based on test functions which locally $C^{2,1}$ and globally bounded. It is an interesting open problem to prove regularity of solutions to $\mathcal H^{s,\pm}_\infty u=0$. If $u$ is globally non-negative, is there a  Harnack inequality? Are viscosity solutions locally H{\"o}lder continuous?

\nocite{*}
\bibliographystyle{abbrv}
\bibliography{Bibliography}

\end{document}